%% file: Calderonrevis.tex
\documentclass{amsart}
\usepackage{stmaryrd}
\usepackage{amssymb,epsfig,xspace}
\usepackage{ifpdf}
\usepackage{booktabs}
\usepackage{color}
\usepackage{bm}
\usepackage{tgpagella}

\usepackage{bbold}
\usepackage{mathrsfs}
\usepackage{esint}

\usepackage{soul}

\newtheorem{theorem}{Theorem}[section]
\newtheorem{lemma}[theorem]{Lemma}
\newtheorem{proposition}[theorem]{Proposition}
\newtheorem{corollary}[theorem]{Corollary}

\theoremstyle{definition}

\theoremstyle{remark}
\newtheorem{remark}[theorem]{Remark}

\numberwithin{equation}{section}

\newcommand{\1}{\mathbb 1}

\newcommand{\eps}{\varepsilon}

\newcommand{\R}{\mathbb R}

\newcommand{\cF}{\mathcal F}

\newcommand{\be}{\begin{equation}}
\newcommand{\ee}{\end{equation}}

\DeclareFontFamily{U}{mathx}{\hyphenchar\font45}
\DeclareFontShape{U}{mathx}{m}{n}{
      <5> <6> <7> <8> <9> <10>
      <10.95> <12> <14.4> <17.28> <20.74> <24.88>
      mathx10
      }{}
\DeclareSymbolFont{mathx}{U}{mathx}{m}{n}
\DeclareFontSubstitution{U}{mathx}{m}{n}
\DeclareMathAccent{\widecheck}{0}{mathx}{"71}
\DeclareMathAccent{\wideparen}{0}{mathx}{"75}

\newcommand{\tria}{{\mathcal T}}
\newcommand{\bbT}{\mathbb{T}}

\newcommand{\cP}{\mathcal{P}}

\newcommand{\identity}{\mathrm{Id}}

\newcommand{\HH}{\mathscr{H}}
\newcommand{\V}{\mathscr{V}}

\newcommand{\B}{\mathscr{B}}
\newcommand{\Ss}{\mathscr{S}}
\newcommand{\W}{\mathscr{W}}
\newcommand{\Y}{\mathscr{Y}}
\newcommand{\ZZ}{\mathscr{Z}}
\newcommand{\nrm}{| \hspace{-.6mm} | \hspace{-.6mm} |}

\DeclareMathOperator{\ran}{ran}
\DeclareMathOperator{\meas}{meas}

\DeclareMathOperator{\supp}{supp}

\DeclareMathOperator{\Span}{span}

\DeclareMathOperator{\diam}{diam}
\DeclareMathOperator{\diag}{diag}

\newcommand{\cL}{\mathcal L}
\newcommand{\Lis}{\cL\mathrm{is}}
\newcommand{\cLis}{\cL\mathrm{is}_c}
\newcommand{\lnrm}{\mathopen{| \! | \! |}}
\newcommand{\rnrm}{\mathclose{| \! | \! |}}

\newcommand{\new}[1]{#1}
\newcommand{\note}[1]{}
\newcommand{\uumlaut}{{\"u}}

{\par\begin{samepage}%
\begin{tabbing}\ttfamily%
 \hspace*{5mm}\=\hspace{3ex}\=\hspace{3ex}\=\hspace{3ex}\=\hspace{3ex}%
\=\hspace{3ex}\=\hspace{3ex}\=\hspace{3ex}\=\hspace{3ex}\kill}%
{\end{tabbing}\end{samepage}}

\title{\new{Uniform preconditioners} for problems of negative order}

\date{\today}

\author{Rob Stevenson, Raymond van Veneti\"{e}}
\address{
Korteweg-de Vries Institute for Mathematics,
University of Amsterdam,
P.O. Box 94248,
1090 GE Amsterdam, The Netherlands}
\email{r.p.stevenson@uva.nl, r.vanvenetie@uva.nl}

\thanks{This work was initiated during the trimester programme ``Multiscale Problems: Algorithms,
Numerical Analysis and Computation'' January-April 2017 at the Hausdorff
Research Institute for Mathematics, Bonn, Germany, whose support is gratefully acknowledged by the first author.
In addition, he has been supported by NSF Grant DMS 1720297.
The second author has been supported by the Netherlands Organization for Scientific Research
(NWO) under contract. no. 613.001.652}



\begin{document}
\begin{abstract} 
\new{Uniform} preconditioners for operators of negative order discretized by (dis)continuous piecewise polynomials of any order are constructed from a boundedly invertible operator of opposite order discretized by continuous piecewise linears. Besides the cost of the application of the latter discretized operator, the other cost of the preconditioner scales linearly with the number of mesh cells.
Compared to earlier proposals, the preconditioner has the following advantages: It does not require the inverse of a non-diagonal matrix; it applies without any mildly grading assumption on the mesh; and it does not require a barycentric refinement of the mesh underlying the trial space.
\end{abstract}

\maketitle

\section{Introduction} 
\subsection{Operator preconditioning}
This paper is about the construction of preconditioners for discretized boundedly invertible linear operators of negative order using the concept of `operator preconditioning' (\cite{138.26}). The idea is to precondition the discretized operator by a discretized operator of opposite order.
This is an appealing idea, but it turns out that in order to get a uniformly well-conditioned system, as well as a preconditioner that can be implemented efficiently, the second discretization has to be carefully chosen dependent on the first one.

For a Hilbert space $\HH$, and a densely embedded reflexive Banach space ${\W \hookrightarrow \HH}$, consider the Gelfand triple
$$
\W \hookrightarrow \HH \simeq \HH' \hookrightarrow \W'.
$$
For $A$ being a boundedly invertible coercive linear operator $\W' \rightarrow \W$, and $\V_\tria \subset \HH$ being a finite dimensional subspace of $\W'$, let $(A_\tria v)(\tilde{v}):=(Av)(\tilde{v})$ ($v, \tilde{v} \in \V_\tria$).
For $B$ being a boundedly invertible coercive linear operator $\W \rightarrow \W'$, and $\W_\tria$ being a finite dimensional subspace of $\W$,
let $(B_\tria w)(\tilde{w}):=(B w)(\tilde{w})$ ($w, \tilde{w} \in \W_\tria$).

A typical example is given by the case that for the boundary $\Gamma$ of some domain, $\HH=L_2(\Gamma)$, $\W=H^{\frac{1}{2}}(\Gamma)$, $A$ is the single layer integral operator \new{arising from the Laplacian}, $B$ \new{is the corresponding} hypersingular integral operator, $\tria$ is a partition from an infinite collection of partitions $\bbT$, $\V_\tria$ is a trial space of discontinuous piecewise polynomials w.r.t.~$\tria$, and $\W_\tria$ is a suitable subspace of $\W$, which thus cannot be equal to $\V_\tria$.
Besides as boundary integral equations, coercive linear operators of order $-1$ also appear in various domain decomposition type methods in the equations for normal fluxes on interfaces.

\new{Although less frequently, coercive linear operators of order $-2$ also appear in the literature (e.g. see \cite{75.067}).}

In order to precondition $A_\tria:\V_\tria \rightarrow \V_\tria'$ with $B_\tria:\W_\tria \rightarrow \W_\tria'$ we need to be able to `identify' $\V_\tria'$ with $\W_\tria$, similar to the identification of $\W''$ with $\W$. Let $\dim \W_\tria=\dim \V_\tria$ and
\be \label{infinfsup}
\inf_{\tria \in \bbT} \inf_{0 \neq v \in \V_\tria} \sup_{0 \neq w \in \W_\tria} \frac{\langle v,w\rangle_\HH}{\|v\|_{\W'} \|w\|_{\W}}>0.
\ee
Then $D_\tria$ defined by $(D_\tria v)(w):=\langle v,w\rangle_\HH$ ($v \in \V_\tria,\, w \in \W_\tria$) is a uniformly boundedly invertible linear map $\V_\tria \rightarrow \W_\tria'$, and so its adjoint $D_\tria'$ is such a map $\V'_\tria \rightarrow \W_\tria$. We conclude that the preconditioned system $D_\tria^{-1} B_\tria (D'_\tria)^{-1} A_\tria$ is uniformly boundedly invertible $\V_\tria \rightarrow \V_\tria$.

Equipping $\V_\tria$ and $\W_\tria$ with bases $\Xi_\tria$ and $\Psi_\tria$, respectively, the matrix representation of the preconditioned system reads as
$\bm{D}_\tria^{-1} \bm{B}_\tria \bm{D}_\tria^{-T} \bm{A}_\tria$, with `stiffness matrices' $\bm{A}_\tria:=(A_\tria \Xi_\tria)(\Xi_\tria)$ and $\bm{B}_\tria:=(B_\tria \Psi_\tria)(\Psi_\tria)$, and `generalized mass matrix' $\bm{D}_\tria:=\langle \Xi_\tria,\Psi_\tria\rangle_\HH$. Regardless of the choice of the bases, the spectral condition number of this matrix is equal to that of $D_\tria^{-1} B_\tria (D'_\tria)^{-1} A_\tria$, and thus uniformly bounded.

After an earlier proposal from \cite{249.16}, the currently commonly followed construction of a suitable pair $(\V_\tria,\W_\tria)$ is the one from \cite{35.8655}.
Here $\V_\tria$ is the space of piecewise constants w.r.t.~a partition $\tria$ of a two-dimensional domain or manifold equipped with the usual basis $\Xi_\tria$, and $\W_\tria$, defined as the span of a collection~$\Psi_\tria$, is a subspace of the space of continuous piecewise linears w.r.t.~a barycentric refinement of $\tria$ constructed by subdividing each triangle into 6 subtriangles by connecting its vertices and midpoints with its barycenter.
In \cite{138.275} the inf-inf-sup condition \eqref{infinfsup} was demonstrated for families of partitions including locally refined ones that satisfy a certain mildly-grading condition from \cite{249.18}.

A problem with the constructions from both \cite{249.16,35.8655} is that the matrix $\bm{D}_\tria$ is \emph{not} diagonal, so that its inverse has to be approximated.
Knowing that $\bm{D}_\tria^{-1} \bm{B}_\tria \bm{D}_\tria^{-T}$ is \emph{not} well-conditioned, because $\bm{A}_\tria$ is not whereas their product is uniformly well-conditioned, the accuracy with which $\bm{D}_\tria^{-1}$ has to be approximated such that it gives rise to a uniform preconditioner increases with an increasing (minimal) mesh-size. 

\subsection{Contributions from this paper}
For the aforementioned $\V_\tria$ and $\Xi_\tria$, in this work a space $\W_\tria$, given as the span of a collection $\Psi_\tria$, will be constructed such that \eqref{infinfsup} is valid, and $\bm{D}_\tria=\langle \Xi_\tria,\Psi_\tria\rangle_\HH$ is \emph{diagonal}. \new{i.e., $\Xi_\tria$ and $\Psi_\tria$ are \emph{biorthogonal}.}
Thanks to \new{both $\Xi_\tria$ and $\Psi_\tria$ being `locally supported',} the corresponding biorthogonal projector \new{onto $\W_\tria$} is \emph{local}, which allows to demonstrate the inf-inf-sup stability \emph{without any mildly grading assumption} on the partitions.

Each function in $\Psi_\tria$ equals a function from the space \new{$\Ss_{\tria,0}^{0,1}$} of continuous piecewise linears\footnote{\new{The subscript $0$  in the notation $\Ss_{\tria,0}^{0,1}$ refers to possible boundary conditions that are incorporated.}}  w.r.t.~$\tria$, plus a linear combination of `bubble functions' from a space denoted as $\B_\tria$. Since the decomposition of $\Ss_{\tria,0}^{0,1} \oplus \B_\tria$ into $\Ss_{\tria,0}^{0,1}$ and $ \B_\tria$ is stable w.r.t.~the $\W$-norm, instead of simply defining $(B_\tria w)(\tilde w):=(B w)(\tilde w)$ \new{($w,\tilde w\in \W_\tria$)}, a suitable boundedly invertible linear operator $B_\tria\colon \W_\tria \rightarrow \W_\tria'$ will be constructed from an \new{invertible} diagonal scaling on the bubble space and a  boundedly invertible linear operator $B_\tria^{\Ss}\colon \Ss_{\tria,0}^{0,1} \rightarrow (\Ss_{\tria,0}^{0,1})'$, e.g.~$(B_\tria^{\Ss} w)(\tilde w):=(B w)(\tilde w)$ \new{with $B$ the hypersingular operator.
Other than in \cite{249.16,35.8655}, by this use of the stable decomposition there is no need to discretize the hypersingular operator on a refinement of $\tria$.
The whole approach relies on \emph{existence} of bubble functions with certain properties, whereas these functions themselves do not enter the implementation.

The total cost of the resulting preconditioner is the sum of the cost of the application of $B_\tria^{\Ss}$ plus cost that scales linearly in $\#\tria$. For $\tria$ being a uniform refinement of some initial coarse partition, a $B_\tria^{\Ss}$ of multi-level type can be found whose cost scales linearly in $\#\tria$ (\cite{34.8}).
Such $B_\tria^{\Ss}$ that also apply on locally refined partitions will be discussed in \cite{249.98}.}

The construction of the biorthogonal collection $\Psi_\tria$, and with that of the preconditioner, is based on a general principle.
It applies in any space dimension, and, as we will see, it applies equally well when $\V_\tria$ is the space of \emph{continuous piecewise linears}.
Higher order discretizations will be covered as well.

The construction applies equally well on manifolds. The coefficients of the functions from $\Psi_\tria$ in terms of functions from $\Ss_{\tria,0}^{0,1}$ and the bubble functions are given as inner products between functions of $\V_\tria$ and $\Ss_{\tria,0}^{0,1}$.
Since in the manifold case, however, generally these inner products cannot be evaluated exactly, we present an alternative slightly modified construction in which the true $L_2$-inner product is replaced by a mesh-dependent one by an element-wise freezing of the Jacobian.
It still yields a uniform preconditioner on general, possibly locally refined partitions, \new{while the same explicit formula} for the expansion coefficients of the functions of $\Psi_\tria$ that was derived in the domain case, \new{now also} applies in the manifold case.

\subsection{Notations}
In this work, by $\lambda \lesssim \mu$ we will mean that $\lambda$ can be bounded by a multiple of $\mu$, independently of parameters which $\lambda$ and $\mu$ may depend on, with the sole exception of the space dimension $d$, or in the manifold case, on the parametrization of the manifold that is used to define the finite element spaces on it. Obviously, $\lambda \gtrsim \mu$ is defined as $\mu \lesssim \lambda$, and $\lambda\eqsim \mu$ as $\lambda\lesssim \mu$ and $\lambda \gtrsim \mu$.


For normed linear spaces $\Y$ and $\ZZ$, in this paper for convenience over $\R$, $\cL(\Y,\ZZ)$ will denote the space of bounded linear mappings $\Y \rightarrow \ZZ$ endowed with the operator norm $\|\cdot\|_{\cL(\Y,\ZZ)}$. The subset of invertible operators in $\cL(\Y,\ZZ)$  with inverses in $\cL(\ZZ,\Y)$
will be denoted as $\Lis(\Y,\ZZ)$.
The \emph{condition number} of a $C \in \Lis(\Y,\ZZ)$ is defined as $\kappa_{\Y,\ZZ}(C):=\|C\|_{\cL(\Y,\ZZ)}\|C^{-1}\|_{\cL(\ZZ,\Y)}$.

For $\Y$ a reflexive Banach space and $C \in \cL(\Y,\Y')$ being \emph{coercive}, i.e.,
$$
\inf_{0 \neq y \in \Y} \frac{(Cy)(y)}{\|y\|^2_\Y} >0,
$$
both $C$ and $\Re(C)\!:= \!\frac{1}{2}(C+C')$ are in $\Lis(\Y,\Y')$ with
\begin{align*}
\|\Re(C)\|_{\cL(\Y,\Y')} &\leq \|C\|_{\cL(\Y,\Y')},\\
\|C^{-1}\|_{\cL(\Y',\Y)} & \leq \|\Re(C)^{-1}\|_{\cL(\Y',\Y)}=\Big(\inf_{0 \neq y \in \Y} \frac{(Cy)(y)}{\|y\|^2_\Y}\Big)^{-1}.
\end{align*}
The set of coercive $C \in \Lis(\Y,\Y')$ is denoted as $\cLis(\Y,\Y')$.
If $C   \in \cLis(\Y,\Y')$, then $C^{-1} \in \cLis(\Y',\Y)$ and $\|\Re(C^{-1})^{-1}\|_{\cL(\Y,\Y')} \leq \|C\|_{\cL(\Y,\Y')}^2 \|\Re(C)^{-1}\|_{\cL(\Y',\Y)}$.

Two countable collections $\Upsilon=(\upsilon_i)_i$ and $\tilde{\Upsilon}=(\tilde{\upsilon}_i)_i$ in a Hilbert space will be called \emph{biorthogonal} when $\langle \Upsilon, \tilde{\Upsilon}\rangle=[\langle \upsilon_j, \tilde{\upsilon}_i\rangle]_{i j}$ is an \emph{invertible diagonal} matrix, and \emph{biorthonormal} when it is the \emph{identity} matrix.

\subsection{Organization}
In Sect.~\ref{Soper} the general principles of operator preconditioning are recalled. 
In Sect.~\ref{Sdomain_case}, it is applied to operators of negative order discretized with \emph{discontinuous} piecewise \new{constants}, first in the domain- and then in the manifold-case.\note{phrase removed}
In Sect.~\ref{Slinears}, the same program is followed for trial spaces of \emph{continuous} piecewise \new{linears}.
\new{In Sect.~\ref{Shigherorder} the results from Sect.~\ref{Sdomain_case}-\ref{Slinears} will be extended to higher order finite element spaces.
This will be done by both applying the operator preconditioning framework directly to the higher order spaces, and by using the preconditioner found for the lowest order case in a subspace correction approach.}
Finally, in Sect.~\ref{Snumerics} we report on some numerical results obtained with the new preconditioners, and compare them with those obtained with the preconditioner from \cite{35.8655,138.275}.

\section{Operator preconditioning} \label{Soper}
The exposition in this section largely follows \cite[Sect.~2]{138.26} closely.
Let $\V$, $\W$ be reflexive Banach spaces. We will search a `preconditioner' $G$ for an $A \in \Lis(\V,\V')$, i.e.~an operator $G \in \Lis(\V',\V)$ (whose application is `easy' compared to that of $A^{-1}$).
\new{It is often useful, e.g. for the application of Conjugate Gradients, when the preconditioner is coercive, i.e., being an operator in $\cLis(\V',\V)$.}
The following result is easily verified.

\begin{proposition} \label{prop10} If $B \in \Lis(\W,\W')$ and $D  \in \Lis(\V,\W')$, then
$$
G:=D^{-1} B (D')^{-1} \in \Lis(\V',\V),
$$
and
\begin{align*}
\|G\|_{\cL(\V',\V)} &\leq \|D^{-1}\|_{\cL(\W',\V)}^2 \|B\|_{\cL(\W,\W')},\\
\|G^{-1}\|_{\cL(\V,\V')} & \leq \|D\|_{\cL(\V,\W')}^2 \|B^{-1}\|_{\cL(\W',\W)}.
\end{align*}
If \new{additionally} $B \in \cLis(\W,\W')$, then $G \in \cLis(\V',\V)$, and 
$$
\|\Re(G)^{-1}\|_{\cL(\V,\V')} \leq \|D\|_{\cL(\V,\W')}^2 \|\Re(B)^{-1}\|_{\cL(\W',\W)}.
$$
\end{proposition}

\begin{remark} \label{closedrange}
We recall that by an application of the \emph{closed range theorem}, $D  \in \cL(\V,\W')$ is in $\Lis(\V,\W')$ if and only if for all $w \in \W$ there exists a $v \in \V$ with $(Dv)(w)\neq 0$, and 
$$
0< \inf_{0 \neq v \in \V} \sup_{0 \neq w \in \W} \frac{(D v)(w)}{\|v\|_\V \|w\|_\W} \quad\Big(= \|D^{-1}\|_{\cL(\W',\V)}^{-1}\Big).
$$ 
\end{remark}

In particular we are interested in finding a preconditioner for an operator $A_\tria \in \Lis(\V_\tria,\V_\tria')$ of the form $G_\tria=D_\tria^{-1} B_\tria (D_\tria')^{-1}$,
where $\V_\tria$ is some \new{\emph{finite dimensional} (finite- or boundary element)} space.
In view of Proposition~\ref{prop10}, for that goal we search some finite dimensional space $\W_\tria$ with
\be \label{equaldims}
\dim \W_\tria=\dim \V_\tria,
\ee
and operators $B_\tria \in \Lis(\W_\tria,\W_\tria')$ and $D_\tria  \in \Lis(\V_\tria,\W_\tria')$.

A typical setting is that, for some reflexive Banach spaces $\V$ and $\W$, and operators
$A \in \cLis(\V,\V')$ and $B \in \cLis(\W,\W')$, we have $\V_\tria \subset \V$ (thus \emph{equipped with~$\|\,\|_\V$}), $(A_\tria u)(v):=(A u)(v)$ \new{($u,v \in \V_\tria$)} and, for a suitable $\W_\tria \subset \W$ (thus \emph{equipped with~$\|\,\|_\W$}), take $(B_\tria w)(z):=(B w)(z)$ \new{($w,z \in \W_\tria$)}. In this case
$A_\tria \in \cLis(\V_\tria,\V_\tria')$ and $B_\tria \in \cLis(\W_\tria,\W_\tria')$ with
\begin{align*}
\|A_\tria\|_{\cL(\V_\tria,\V_\tria')} \leq \|A\|_{\cL(\V,\V')}&,\,
\|\Re(A_\tria)^{-1}\|_{\cL(\V_\tria',\V_\tria)} \leq \|\Re(A)^{-1}\|_{\cL(\V',\V)},\\
\|B_\tria\|_{\cL(\W_\tria,\W_\tria')} \leq \|B\|_{\cL(\W,\W')}&,\,
\|\Re(B_\tria)^{-1}\|_{\cL(\W_\tria',\W_\tria)} \leq \|\Re(B)^{-1}\|_{\cL(\W',\W)}.
\end{align*}
An \new{obvious} construction of a suitable $D_\tria$ is discussed in the next proposition.

\begin{proposition}[Fortin projector (\cite{75.22})] \label{fortin} For some $D  \in \Lis(\V,\W')$, let \new{$D_\tria \in \cL(\V_\tria,\W_\tria')$ be defined by
$(D_\tria v)(w):=(D v)(w)$.} Then 
\begin{align} \nonumber
\|D_\tria\|_{\cL(\V_\tria,\W_\tria')} &\leq  \|D\|_{\cL(\V,\W')}.
\intertext{Assuming \eqref{equaldims}, \new{additionally one has} $D_\tria \in \Lis(\V_\tria,\W_\tria')$ \emph{if}, and \new{for} $\W$ \new{being} a Hilbert space, \emph{only if} there exists a projector $P_\tria \in \cL(\W,\W)$ onto $\W_\tria$ with \mbox{$(D \V_\tria )((\identity-P_\tria)\W)=0$,} \new{in which case}}
\label{bound}
\|D_\tria^{-1}\|_{\cL(\W_\tria',\V_\tria)} &\leq \|P_\tria\|_{\cL(\W,\W)} \|D^{-1}\|_{\cL(\W',\V)}.
\end{align}
\end{proposition}

\begin{proof} The first statement is obvious. Now let us assume existence of a (Fortin) projector $P_\tria$. Then for $v_\tria \in \V_\tria$,
\begin{align*}
 \|D^{-1}\|_{\cL(\W',\V)}^{-1} \|v_\tria\|_\V & \leq \sup_{0 \neq w \in \W} \frac{(D v_\tria)(w)}{\|w\|_\W}=
 \sup_{0 \neq w \in \W} \frac{(D v_\tria)(P_\tria w)}{\|w\|_\W}\\
 &  \leq \|P_\tria\|_{\cL(\W,\W)} \sup_{0 \neq w_\tria \in \W_\tria} \frac{(D v_\tria)(w_\tria)}{\|w_\tria\|_\W},
 \end{align*}
 which together with Remark~\ref{closedrange} and \eqref{equaldims} shows that $D_\tria \in \Lis(\V_\tria,\W_\tria')$, in particular \eqref{bound}.
 
 Conversely (cf. \cite[Remark 4.9]{31}), assume $D_\tria \in \Lis(\V_\tria,\W_\tria')$, \new{and let $\W$ be a Hilbert space}. Then given $w \in \W$, let $w_\tria$ be the first component of the solution $(w_\tria,v_\tria) \in \W_\tria \times \V_\tria$ of the well-posed saddle point problem
 \begin{alignat*}{4}
& \langle w_\tria,z_\tria \rangle_\W+ (D_\tria v_\tria)(z_\tria) && = &&  \langle w,z_\tria\rangle_\W &\quad &(z_\tria \in \W_\tria),\\
& (D_\tria u_\tria)(w_\tria) && = && (D_\tria u_\tria)(w) &\quad &(u_\tria \in \V_\tria).
 \end{alignat*}
Then $P_\tria:=w \mapsto w_\tria$ is a valid Fortin projector.
\end{proof}

In applications, one usually has a \emph{family} of spaces $\V_\tria$ and aims at a \emph{uniform} preconditioner $G_\tria$. 
In the setting of Proposition~\ref{fortin} it means that one searches a Fortin projector $P_\tria$ such that $ \|P_\tria\|_{\cL(\W,\W)} $ is \emph{uniformly} bounded.

\subsection{Implementation}
Given a finite collection $\Upsilon=\new{\{\upsilon\}}$ in a linear space, we set the \emph{synthesis operator} 
$$
\cF_\Upsilon:\R^{\# \Upsilon} \rightarrow \Span\Upsilon\colon {\bf c} \mapsto {\bf c}^\top \Upsilon:=\sum_{\upsilon \in \Upsilon} c_\upsilon \upsilon.
$$
Equipping $\R^{\# \Upsilon}$ with the Euclidean scalar product $\langle\,,\,\rangle$, and identifying $(\R^{\# \Upsilon})'$ with~$\R^{\# \Upsilon}$ using the corresponding Riesz map, we infer that the adjoint of $\cF_\Upsilon$, known as the \emph{analysis operator}, satisfies
$$
\cF'_\Upsilon:(\Span\Upsilon)' \rightarrow \R^{\# \Upsilon}:f \mapsto f(\Upsilon):=[f(\upsilon)]_{\upsilon \in \Upsilon}.
$$
A collection $\Upsilon$ is a \emph{basis} for its span when $\cF_\Upsilon\in \Lis(\R^{\# \Upsilon},\Span\Upsilon)$ (and so $\cF'_\Upsilon \in \Lis((\Span\Upsilon)' ,\R^{\# \Upsilon})$.)

Now let $\Xi_\tria=\new{\{\xi\}}$ and $\Psi_\tria=\new{\{\psi\}}$ be bases for $\V_\tria$ and $\W_\tria$, respectively.
Then in coordinates the preconditioned system reads as
$$
\cF_{\Xi_\tria}^{-1} G_\tria A_\tria \cF_{\Xi_\tria}= \bm{G}_\tria \bm{A}_\tria:=\bm{D}_\tria^{-1} \bm{B}_\tria \bm{D}_\tria^{-\top} \bm{A}_\tria,
$$
where
$$
\bm{A}_\tria:=\cF'_{\Xi_\tria} A_\tria \cF_{\Xi_\tria},\quad
\bm{B}_\tria:=\cF'_{\Psi_\tria} B_\tria \cF_{\Psi_\tria},\quad
\bm{D}_\tria:=\cF'_{\Psi_\tria} D_\tria \cF_{\Xi_\tria}.
$$

By identifying a map in $\cL(\R^{\# \Xi_\tria},\R^{\# \Xi_\tria})$ with a ${\# \Xi_\tria} \times {\# \Xi_\tria}$ matrix by equipping~$\R^{\# \Xi_\tria}$ with the canonical basis \new{$\{\bm{e}_i\}$}, \new{and by enumerating the elements of $\Xi_\tria$} one has
$$
(\bm{A}_\tria)_{\new{ij}}=\langle \cF'_{\Xi_\tria} A_\tria \cF_{\Xi_\tria} \bm{e}_{\new{j}}, \bm{e}_{\new{i}}\rangle=(A_\tria \cF_{\Xi_\tria}  \bm{e}_{\new{j}})(\cF_{\Xi_\tria}  \bm{e}_{\new{i}})=(A_\tria  \xi_{\new{j}})(\xi_{\new{i}}),
$$
and similarly,
$$
(\bm{B}_\tria)_{\new{ij}}=(B_\tria\psi_{\new{j}})(\psi_{\new{i}}), \quad (\bm{D}_\tria)_{\new{ij}}=(D_\tria \xi_{\new{j}})(\psi_{\new{i}}).
$$
\new{Preferably $\bm{D}_\tria$} is such that its inverse can be applied in linear complexity, as is the case when $\bm{D}_\tria$ is \emph{diagonal}.

\begin{remark}
\new{Using $\sigma(\,)$ and $\rho(\,)$ to denote the spectrum and spectral radius of an operator,}
clearly $\sigma(\bm{G}_\tria \bm{A}_\tria)=\sigma(G_\tria   A_\tria)$.  So for the spectral condition number we have
$$ 
\kappa_S(\bm{G}_\tria \bm{A}_\tria):=\rho(\bm{G}_\tria \bm{A}_\tria)\rho((\bm{G}_\tria \bm{A}_\tria)^{-1}) \leq \kappa_{\V_\tria,\V_\tria}(G_\tria A_\tria),
$$
which thus holds true \emph{independently} of the choice of the basis $\Xi_\tria$ for $\V_\tria$.
Furthermore, in view of an application of Conjugate Gradients, \new{if} $A_\tria$ and $B_\tria$ are coercive and \emph{self-adjoint}, then $\bm{A}_\tria$ and $\bm{G}_\tria$ are \new{positive definite and symmetric}.
Equipping $\R^{\dim \V_\tria}$ with $\lnrm\cdot\rnrm:=\|(\bm{G}_\tria)^{-\frac{1}{2}}\cdot\|$ or $\lnrm\cdot\rnrm:=\|(\bm{A}_\tria)^{\frac{1}{2}}\cdot\|$, in that case we have
$$
\kappa_{(\R^{\dim \V_\tria},\lnrm\cdot\rnrm),(\R^{\dim \V_\tria},\lnrm\cdot\rnrm)}(\bm{G}_\tria \bm{A}_\tria)=\kappa_S(\bm{G}_\tria \bm{A}_\tria)
.
$$
\end{remark}

\section{Preconditioning an operator of negative order discretized by piecewise constants} \label{Sdomain_case}
For  a bounded polytopal domain $\Omega \subset \R^d$, a measurable, closed, possibly empty $\gamma \subset \partial\Omega$, and an $s \in [0,1]$, we take
 $$
\framebox{$\W:=[L_2(\Omega),H^1_{0,\gamma}(\Omega)]_{s,2},\quad \V:=\W'$},
$$
where $H^1_{0,\gamma}(\Omega)$ is the closure in $H^1(\Omega)$ of the $C^\infty(\Omega) \cap H^1(\Omega)$ functions that vanish at $\gamma$.\footnote{In the domain case, it is easy to generalize the results to Sobolev spaces with smoothness index $s \in [0,\frac{3}{2})$, or even to $s \in (-\frac{1}{2},\frac{3}{2})$.}
The role of $D \in \Lis(\V,\W')$ in Proposition~\ref{fortin} is going to be played by \new{the unique extension to $\V \times \W$ of} the duality pairing
$$
(D v)(w):=\langle v,w\rangle_{L_2(\Omega)},
$$
which satisfies $\|D\|_{\cL(\V,\W')}=\|D^{-1}\|_{\cL(\W',\V)}=1$.  

 Let $(\tria)_{\tria \in \bbT}$ be a family of \emph{conforming} partitions of $\Omega$ into (open) \emph{uniformly shape regular} $d$-simplices, where we assume that $\gamma$ is the (possibly empty) union of $(d-1)$-faces of $T \in \tria$.
 Thanks to the conformity and the uniform shape regularity, for $d>1$ we know that neighbouring $T,T' \in \tria$, i.e.~$\overline{T} \cap \overline{T'} \neq \emptyset$, have uniformly comparable sizes.
 For $d=1$, we impose this uniform `\emph{$K$-mesh property}' explicitly.\footnote{For our convenience, throughout this paper we consider trial spaces w.r.t.~conforming partitions
 into uniformly shape regular $d$-simplices.
 It will however become clear that families of non-conforming partitions into uniformly shape regular $d$-simplices or hyperrectangles
 that satisfy a uniform $K$-mesh property can be dealt with as well.}
 
For $\tria \in \bbT$,  we define \new{$N^0_\tria$} as the set of vertices of $\tria$ that are not on $\gamma$, and for $\nu \in N^0_\tria$ we set its \emph{valence}
$$
d_{\tria,\nu}:=\#\{T \in \tria\colon \nu \in \overline{T}\}.
$$
For $T \in \tria$, and with $N_T$ denoting the set of its vertices, we set $\new{N^0_{\tria,T}}:=N^0_\tria \cap N_T$, and define $h_T:=|T|^{1/d}$.

 We take
 $$
\framebox{$\V_\tria=\Ss^{-1,0}_\tria:=\{u \in L_2(\Omega)\colon u|_T \in \cP_0 \,(T \in \tria)\} \subset \V$},
$$
and, as a first ingredient in the \new{forthcoming} construction of a suitable $\W_\tria$, \new{define the space of continuous piecewise linears, zero on $\gamma$, by}
$$
\Ss^{0,1}_{\tria,0}:=\{u \in H^1_{0,\gamma}(\Omega)\colon u|_T \in \cP_1 \,(T \in \tria)\},
 $$
 equipped with the usual bases
 \be \label{20}
 \Xi_\tria=\{\xi_T\colon T \in \tria\}, \quad \Phi_\tria=\{\phi_{\tria,\nu}\colon \nu \in N^0_\tria\},
 \ee
 respectively, defined by
 \be \label{21}
 \xi_T:=\left\{\begin{array}{cl} 1 & \text{on } T,\\ 0 & \text{on } \Omega \setminus T,\end{array} \right.\quad
 \phi_{\tria,\nu}(\nu')=\delta_{\nu,\nu'} \quad (\nu,\nu'\in N^0_\tria).
 \ee
 
\subsection{Construction of $\W_\tria$ and $\bm{D}_\tria$.}
\new{Aiming at the construction of a (uniform) preconditioner $G_\tria \in \cLis(\V_\tria',\V_\tria)$ using the framework of operator preconditioning,} we are going to construct a collection $\Psi_\tria \subset H^1_{0,\gamma}(\Omega)$ that is biorthogonal to $\Xi_\tria$, 
\new{whose elements are `locally supported', and for which}
$$
\framebox{$\W_\tria:=\Span \Psi_\tria \subset \W$}
$$
has an `approximation property'. These \new{three} properties of $\Psi_\tria$ will allow us to construct a suitable Fortin projector, and they will give rise to a matrix $\bm{D}_\tria$ \new{that is \emph{diagonal}}.

The construction of $\Psi_\tria$ builds on two collections $\Theta_\tria$ and $\Sigma_\tria$ \new{of `locally supported' functions} in  $H^1_{0,\gamma}(\Omega)$ whose cardinalities are equal to that of $\Xi_\tria$, the first being biorthogonal to $\Xi_\tria$, and the second \new{whose span has an `approximation property' and is} inside $\Ss^{0,1}_{\tria,0}$.\note{Less relevant footnote removed}
 
 Let $\Theta_\tria=\{\theta_T\colon T \in \tria\} \subset H^1_{0,\gamma}(\Omega)$  be such \new{that $\theta_T\geq 0$, $\supp \theta_T \subset \overline{T}$,}
 \be \label{C1}
 \langle \theta_T,\xi_{T'}\rangle_{L_2(\Omega)}  \eqsim \delta_{T T'} \|\theta_T\|_{L_2(\Omega)} \|\xi_{T'}\|_{L_2(\Omega)},\quad(T,T'\in \tria),
\ee
\new{and, for convenience only, that is scaled such that
\be\label{scale}
 \langle \theta_T,\xi_{T}\rangle_{L_2(\Omega)} =|T|.
\ee

One obvious possible construction of such $\Theta_\tria$ is to take $\theta_T$ to be the `bubble function' defined by $\theta_T(x)=\frac{(2 d+1)!}{d} \prod_{i=1}^{d+1} \lambda_i(x)$ for $x \in T$, and zero elsewhere, where $(\lambda_1(x),\cdots,\lambda_{d+1}(x))$ are the barycentric coordinates of $x$ w.r.t. $T$
(see e.g. \cite{308.5} for \eqref{scale}).
 Two forthcoming (harmless) conditions \eqref{14} and \eqref{stab} on $\Theta_\tria$ will be satisfied as well by the above specification of  $\theta_T$.
 
 Another, equally suited construction is, after making a uniform barycentric refinement of $\tria$, to take $\theta_T$ as a the continuous piecewise linear hat function associated to the barycenter of $T$, multiplied by a factor $d+1$.
 
 We emphasize that the resulting preconditioner will \emph{not} depend on the actual construction of $\Theta_\tria$,
 but that only \emph{existence} of a collection with the aforementioned properties is relevant.}
 
   Defining $\Sigma_\tria=\{\sigma_{\tria,T}\colon T \in \tria\} \subset \Ss^{0,1}_{\tria,0}$ by
 $$
\sigma_{\tria,T} :=\sum_{\nu \in N^0_{\tria,T} }d_{\tria,\nu}^{-1} \phi_{\tria,\nu},
$$
 we have
 $$
 \sum_{T \in \tria} \sigma_{\tria,T}=\sum_{\nu \in N^0_\tria} \phi_{\tria,\nu},
 $$
 being equal to the constant function $\1$ on $\Omega \setminus \cup_{\{T \in \tria\colon \overline{T} \cap \gamma \neq \emptyset\}} \overline{T}$,
 \new{which yields the aforementioned `approximation property' (cf. footnote~\ref{approxprop}).}

 We now define
$$
\Psi_\tria:=\{\psi_{\tria,T} \colon T \in \tria\} \subset \Ss^{0,1}_{\tria,0}  \oplus  \Span \Theta_\tria,
 $$
 by
\be \label{defpsi}
\framebox{$\displaystyle  \psi_{\tria,T} := \sigma_{\tria,T}+\frac{\langle \1-\sigma_{\tria,T},\xi_T\rangle_{L_2(\Omega)}}{\langle \theta_T,\xi_T\rangle_{L_2(\Omega)}} \theta_T-\sum_{T' \in \tria\setminus\{T\}} \frac{\langle \sigma_{\tria,T},\xi_{T'}\rangle_{L_2(\Omega)}}{\langle \theta_{T'},\xi_{T'}\rangle_{L_2(\Omega)}} \theta_{T'}$},
 \ee
 The third term at the right-hand side corrects $\sigma_{\tria,T}$ such that it becomes orthogonal to $\xi_{T'}$ for $T' \neq T$, whereas the second term \new{ensures that the $\psi_{\tria,T}$ sum up to $\1$, possibly except on a strip along the Dirichlet boundary:}

  \begin{lemma} \label{prop1} It holds that
 \be \label{5}
\sum_{T \in \tria} \psi_{\tria,T} =\sum_{T \in \tria} \sigma_{\tria,T}+\sum_{T \in \tria} \frac{\langle\1-\sum_{T' \in \tria} \sigma_{\tria,T'},\xi_T\rangle_{L_2(\Omega)}}{\langle \theta_T,\xi_T\rangle_{L_2(\Omega)}} \theta_T,
\ee
and
 \be \label{6}
 \langle \Xi_\tria,\Psi_\tria \rangle_{L_2(\Omega)}=\diag\{\langle \1,\xi_T\rangle_{L_2(\Omega)}\colon T \in \tria\}.
\ee
 \end{lemma}
 
 \begin{proof}
 Writing $\1-\sigma_{\tria,T}=\sum_{T' \in \tria \setminus\{T\}} \sigma_{\tria,T'}+(\1-\sum_{T'\in \tria} \sigma_{\tria,T'})$, \eqref{5} follows from \eqref{defpsi}
\new{by using that
$$
\sum_{T \in \tria}  \sum_{T' \in \tria  \setminus\{T\}}  \frac{\langle\sigma_{\tria,T'},\xi_T\rangle_{L_2(\Omega)}}{\langle \theta_T,\xi_T\rangle_{L_2(\Omega)}} \theta_T
 -
 \sum_{T \in \tria} \sum_{T' \in \tria  \setminus\{T\}} \frac{\langle\sigma_{\tria,T},\xi_{T'}\rangle_{L_2(\Omega)}}{\langle \theta_{T'},\xi_{T'}\rangle_{L_2(\Omega)}} \theta_{T'}=0.
 $$}
 The biorthonormality of $\Xi_\tria$ and $\{\theta_T/\langle \theta_T,\xi_T\rangle_{L_2(\Omega)}\colon T \in \tria\}$ shows \eqref{6}.
 \end{proof}

\new{By} expanding $\sigma_{\tria,T}$ in terms of the nodal basis, \new{and by} using $\int_T \phi_{\tria,\nu}\,dx=\frac{|T|}{d+1}$ \new{and the normalization \eqref{scale}}, we arrive at the \new{explicit} expression
\be
\label{psi1}
 \framebox{$\displaystyle \psi_{\tria,T} := \sum_{\nu \in N^0_{\tria,T}} \hspace*{-0.4em}d_{\tria,\nu}^{-1} \phi_{\tria,\nu}+
\big( 1-{\textstyle \frac{1}{d+1}}\hspace*{-0.8em}\sum_{\nu \in N^0_{\tria,T}} d_{\tria,\nu}^{-1}\big) \theta_T
 - \hspace*{-0.4em}\sum_{T' \in \tria\setminus\{T\}} \hspace*{-0.6em}\big({\textstyle \frac{1}{d+1}}\hspace*{-1.3em}\sum_{\nu \in N^0_{\tria,T} \cap N^0_{\tria,T'}} \hspace*{-0.8em}d_{\tria,\nu}^{-1}\big) \theta_{T'}$},
\ee
see Figure~\ref{fig1} for an illustration. 
\begin{figure}[ht]
\begin{center}
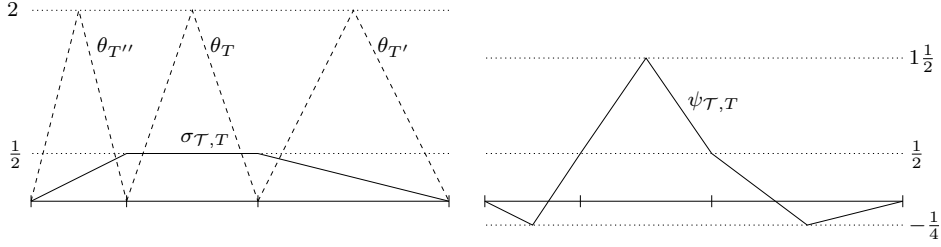
\end{center}
\caption{$\psi_{\tria,T}$ in one dimension (with bubbles constructed using a barycentric refinement).}
\label{fig1}
\end{figure}

 As a consequence of \new{$\langle \Xi_\tria,\Psi_\tria \rangle_{L_2(\Omega)}$ being invertible}, the biorthogonal `Fortin' projector $P_\tria:L_2(\Omega) \rightarrow H^1_{0,\gamma}(\Omega)$ with $\ran P_\tria=\W_\tria$ and $\ran (\identity -P_\tria)=\V_\tria^{\perp_{L_2(\Omega)}}$ exists, and is, \new{thanks to \eqref{6}}, given by
 $$
 P_\tria u=\sum_{T \in \tria} \frac{\langle u,\xi_T \rangle_{L_2(\Omega)}}{\langle \1,\xi_T\rangle_{L_2(\Omega)}} \psi_{\tria,T}.
 $$
 
\new{To prepare for the proof of (uniform) boundedness of $P_\tria$,} we list a few properties of the collections $\Xi_\tria$, $\Theta_\tria$ and $\Sigma_\tria$. For $T \in \tria$, we set $\omega^{(0)}_\tria(T):=\overline{T}$, and for $i=0,1,\ldots$, define \new{the `rings'}
 $$
 R^{(i+1)}_\tria(T):=\{T' \in \tria\colon \overline{T'} \cap \omega^{(i)}_\tria(T)\neq \emptyset\}, \quad \omega^{(i+1)}_\tria(T):=\cup_{T' \in R^{(i+1)}_\tria(T)} \overline{T'}.
 $$
 It holds that
\new{\begin{align} \label{e1}
 {\rm (a)}\,& \supp \xi_T \subset \omega^{(0)}_\tria(T),\quad
 {\rm (b)}\, \supp \sigma_{\tria,T} \subset \omega^{(1)}_\tria(T),\quad
 {\rm (c)}\, \supp \theta_T \subset \omega^{(0)}_\tria(T),\\ \label{e2}
 &\|\sigma_T\|_{H^k(\Omega)} \lesssim h_T^{d/2-k}\quad (k \in \{0,1\}),\\ \label{e3}
& \|\theta_T\|_{H^1(\Omega)} \lesssim h_T^{-1} \|\theta_T\|_{L_2(\Omega)},
\intertext{whilst moreover $\Xi_\tria$ is such that} \label{e4}
&\langle \1,\xi_T\rangle_{L_2(\Omega)} \eqsim h_T^{d/2} \|\xi_T\|_{L_2(\Omega)}.
 \end{align}
 
 From \eqref{e1}(b,c), we obtain that 
 \be \label{e5}
 \supp \psi_{\tria,T} \subset \omega^{(1)}_\tria(T).
 \ee
 By using \eqref{C1}, \eqref{e3}, \eqref{e2} and \eqref{e1}(a) we infer that for 
 $k \in \{0,1\}$
 $$
\Big\|\frac{\langle \sigma_{\tria,T}-\delta_{T T'}\1,\xi_{T'}\rangle_{L_2(\Omega)}}{\langle \theta_{T'},\xi_{T'}\rangle_{L_2(\Omega)}} \theta_{T'} \Big\|_{H^k(\Omega)}
\lesssim h_T^{-k} \|\sigma_{\tria,T}-\delta_{T T'}\1\|_{L_2(\supp \xi_{T'})}
\lesssim h_T^{d/2-k},
 $$
which, by again using \eqref{e2} and \eqref{e3}, shows that
 \be \label{e6}
 \|\psi_{\tria,T}\|_{H^k(\Omega)} \lesssim h_T^{d/2-k} \quad (k \in \{0,1\}).
 \ee}
  \note{Remark removed}
 
 \begin{theorem} \label{theorem1} It holds that $\sup_{\tria \in \bbT} \|P_\tria\|_{\cL(\W,\W)} <\infty$.
 \end{theorem}
 
 \begin{proof} 
 \new{From \eqref{e5}, \eqref{e6}, and \eqref{e4},} we have
 \begin{align} \nonumber
 \|P_\tria u\|_{H^k(T)} &\leq \sum_{T' \in R^{(1)}_\tria(T)} \|\psi_{\tria,T'}\|_{H^k(\new{\Omega})} \frac{\|u\|_{L_2(T')}\|\xi_{T'}\|_{L_2(\new{\Omega})}}{|\langle \1,\xi_{T'}\rangle_{L_2(\Omega)}|} \\ \label{10}
& \lesssim h_T^{-k} \|u\|_{L_2(\omega^{(1)}_\tria(T))} \quad(k \in \{0,1\}),
 \end{align}
which in particular shows that
 \be \label{12}
 \sup_{\tria \in \bbT} \|P_\tria\|_{\cL(L_2(\Omega),L_2(\Omega))} <\infty.
 \ee
 
 To continue, we revisit the construction of $\W_\tria$ and its basis $\Psi_\tria$
 by temporarily including in $N^0_\tria$ also vertices of $\tria$ that \new{lie} on the Dirichlet boundary $\gamma$.
\new{Denoting the extended set of vertices by $N_\tria$}, consequently for the `new' $\psi_{\tria,T}$, \eqref{5} shows that
 \be \label{9}
 \sum_{T \in \tria} \psi_{\tria,T}=\sum_{\nu  \in N_\tria} \phi_{\tria,\nu}=\1 \text{ on }\Omega.
 \ee
 
 For any $\nu \in N_\tria$, we select \new{a} $(d-1)$-face $e$ of a $T \in \tria$ with $\nu \in e$ and $e \subset \gamma$ if $\nu \in \gamma$, and define the functional 
 $$
 g_{\tria,\nu}(u):=\fint_e u \, ds.
 $$
 By the trace theorem and homogeneity arguments (see e.g \cite[(3.6)]{247.2}), one infers that
 \be \label{13}
 | g_{\tria,\nu}(u)|\leq |e|^{-1} \|u\|_{L_1(e)} \lesssim h_T^{-\frac{d}{2}} \|u\|_{L_2(T)}+h_T^{-\frac{d}{2}+1} |u|_{H^1(T)}.
 \ee
 
 For $T \in \tria$, we select a $\nu \in N_T$ with $\nu \in \gamma$ if $\overline{T} \cap \gamma \neq \emptyset$, and define
$$
 g_{\tria,T}:=g_{\tria,\nu},
 $$
 and a Scott-Zhang (\cite{247.2}) type \new{quasi-}interpolator $\Pi_\tria:H^1(\Omega) \rightarrow \W_\tria$\footnote{\new{The existence of such a $\Pi_\tria$ which satisfies an estimate of type \eqref{11} for $k=0$ can be used as a definition of a (lowest order) approximation property of $\W_\tria$.}\label{approxprop}} by
 $$
 \Pi_\tria u=\sum_{T \in \tria} g_{\tria,T}(u) \psi_{\tria,T}.
 $$
It satisfies
 $$
\|\Pi_\tria u\|_{H^k(T)} \lesssim h_T^{-k} \|u\|_{L_2(\omega^{(2)}_\tria(T))}+h_T^{1-k}|u|_{H^1(\omega^{(2)}_\tria(T))} \quad(k \in \{0,1\}).
 $$
 
 Invoking \eqref{9} and using that $g_{\tria,T}(\1)=1$, we infer that for $k \in \{0,1\}$
 \be \label{11}
 \begin{split}
\hspace*{-5.5em} \|(\identity-\Pi_\tria) u\|_{H^k(T)}&=\inf_{p \in \cP_0} \|(\identity-\Pi_\tria) (u-p)\|_{H^k(T)}\\
 & \leq \inf_{p \in \cP_0} \|u-p\|_{H^k(T)}\!+\!h_T^{-k} \|u-p\|_{L_2(\omega^{(2)}_\tria(T))}\!+\!h_T^{1-k} |u|_{H^1(\omega^{(2)}_\tria(T))}\hspace*{-5.5em} \\
& \eqsim \inf_{p \in \cP_0} h_T^{-k}\|u-p\|_{L_2(\omega^{(2)}_\tria(T))}+h_T^{1-k}|u|_{H^1(\omega^{(2)}_\tria(T))}\\
 &\eqsim h_T^{1-k}|u|_{H^1(\omega^{(2)}_\tria(T))}
  \end{split}
\ee
by an application of the Bramble-Hilbert lemma (cf. \cite[(4.2)]{247.2}).

Noting that the `new' $\psi_{\tria,T}$ differs only from the `old', original one when $\overline{T} \cap \gamma \neq \emptyset$, and that for those $T$ and $u \in H^1_{0,\gamma}(\Omega)$ it holds that $g_{\tria,T}(u)=0$, we conclude that $\ran \Pi_\tria |_{H^1_{0,\gamma}(\Omega)}$ is included in the original space $\W_\tria$, which we consider again from here on.
Using that $P_\tria$ is a projector onto this $\W_\tria$, for $u \in H^1_{0,\gamma}(\Omega)$ writing 
$P_\tria u=\Pi_\tria u+P_\tria(\identity-\Pi_\tria)u$, using  \eqref{10} and \eqref{11} for $k\in \{0,1\}$ we arrive at
\begin{align*}
\|P_\tria u\|_{H^1(T)} &\lesssim \|\Pi_\tria u\|_{H^1(T)}+h_T^{-1} \|(\identity-\Pi_\tria) u\|_{L_2(\omega_\tria^{(1)}(T))}\\
& \lesssim \|u\|_{H^1(\omega^{(2)}_\tria(T))}+h_T^{-1} \|(\identity-\Pi_\tria) u\|_{L_2(\omega_\tria^{(1)}(T))}\\
& \lesssim \|u\|_{H^1(\omega^{(3)}_\tria(T))},
\end{align*}
and consequently,
$$
 \sup_{\tria \in \bbT} \|P_\tria\|_{\cL(H^1_{0,\gamma}(\Omega),H^1_{0,\gamma}(\Omega)))} <\infty.
$$
In combination with \eqref{12}, the proof is completed by an application of the Riesz-Thorin interpolation theorem.
 \end{proof}
 
\new{From Proposition~\ref{fortin} and Theorem~\ref{theorem1} we conclude the following:
\begin{corollary} For $D_\tria\colon \V_\tria \rightarrow \W'_\tria$ defined by $(D_\tria v)(w):=(Dv)(w)=\langle v,w\rangle_{L_2(\Omega)}$, it holds that
$D_\tria \in \Lis(\V_\tria,\W_\tria')$ with $\|D_\tria\|_{\cL(\V_\tria,\W_\tria')} \leq 1$ and $\sup_{\tria \in \bbT} \|D^{-1}_\tria\|_{\cL(\W'_\tria,\V_\tria)}\leq \sup_{\tria \in \bbT} \|P_\tria\|_{\cL(\W_\tria,\W_\tria)}<\infty$.
 \end{corollary}
  This} result is thus valid 
  \emph{without any additional assumptions on the mesh grading}.
  The latter is a consequence of the fact that we were able to equip $\V_\tria$ and $\W_\tria$ with local biorthogonal bases. 
  (Compare \cite[eq. (2.30)]{249.18} for conditions on the mesh grading without having local biorthogonal bases).
 Additionally,  the biorthogonality has the important advantage of the matrix 
$$
 \framebox{$\bm{D}_\tria=\langle \Xi_\tria,\Psi_\tria \rangle_{L_2(\Omega)}=\diag\{|T| \colon T \in \tria\}$}
$$
 being \emph{diagonal}. \medskip

\new{Before we discuss in \S\ref{SB} the construction of $B_\tria \in \cLis(\W_\tria,\W_\tria')$, being the last ingredient of our preconditioner, in the following subsection \S\ref{Smani} we revisit the construction of $\W_\tria$ and $\bm{D}_\tria$ in the manifold case.}

  \note{Remark removed}

 \subsection{Construction of $\W_\tria$ and $\bm{D}_\tria$ in the manifold case} \label{Smani}
 Let $\Gamma$ be a compact 
 $d$-dimensional Lipschitz, piecewise smooth manifold in $\R^{d'}$ \new{for some $d' \geq d$} with or without boundary $\partial\Gamma$.
 For some closed measurable $\gamma \subset \partial\Gamma$ and $s \in [0,1]$, let
 $$
\W:=[L_2(\Gamma),H^1_{0,\gamma}(\Gamma)]_{s,2},\quad \V:=\W'.
$$
 We assume that $\Gamma$ is given as the essentially disjoint union of $\cup_{\new{i}=1}^p \overline{\new{\chi}_i(\Omega_i)}$, with, for $1 \leq i \leq p$,
$\new{\chi}_i\colon \R^d \rightarrow \R^{d'}$ being some smooth regular parametrization, and $\Omega_i \subset \R^d$ an open polytope.
 W.l.o.g. assuming that for $i \neq j$, $\overline{\Omega}_i \cap \overline{\Omega}_j=\emptyset$, we define
 $$
 \new{\chi}\colon \Omega:=\cup_{i=1}^p \Omega_i \rightarrow \cup_{i=1}^p \new{\chi}_i(\Omega_i) \text{ by } \new{\chi}|_{\Omega_i}=\new{\chi}_i.
 $$

Let $\bbT$ be a family of conforming partitions $\tria$ of $\Gamma$ \new{into `panels'} such that, for $1 \leq i \leq p$,  $\new{\chi}^{-1}(\tria) \cap \Omega_i$ is a uniformly shape regular conforming partition of $\Omega_i$ into $d$-simplices (that for $d=1$ satisfies a uniform $K$-mesh property).
We assume that $\gamma$ is a (possibly empty) union of `faces' of $T \in \tria$ (i.e., sets of type $\new{\chi}_i(e)$, where $e$ is a $(d-1)$-dimensional face of $\new{\chi}_i^{-1}(T)$).

As in Sect.~\ref{Sdomain_case}, for $\tria \in \bbT$, we define $N^0_\tria$ as the set of vertices of $\tria$ that are not on $\gamma$, set
$d_{\tria,\nu}:=\#\{T \in \tria\colon \nu \in \overline{T}\}$, and for $T \in \tria$, define $h_T:=|T|^{1/d}$ and
$N^0_{\tria,T}:=N^0_\tria \cap N_T$, with $N_T$ being the set of the vertices of $T$.

  We set
 \begin{align*}
\V_\tria=\Ss^{-1,0}_\tria&:=\{u \in L_2(\Gamma)\colon u \circ \new{\chi} |_{\new{\chi}^{-1}(T)} \in \cP_0 \,\,(T \in \tria)\} \subset \V,\\
\Ss^{0,1}_{\tria,0}&:=\{u \in H^1_{0,\gamma}(\Gamma)\colon u \circ \new{\chi} |_{\new{\chi}^{-1}(T)} \in \cP_1 \,\,(T \in \tria)\},
 \end{align*}
equipped  with $\Xi_\tria=\{\xi_T\colon T \in \tria\}$ and  $\Phi_\tria=\{\phi_{\tria,\nu}\colon \nu \in N^0_\tria\}$, respectively, defined by
$ \xi_T := 1$ on $T$, $ \xi_T := 0$ elsewhere, and $\phi_{\tria,\nu}(\nu')=\delta_{\nu,\nu'}$ ($\nu,\nu'\in N^0_\tria$).
Furthermore, we define $\Sigma_\tria=\{\sigma_{\tria,T} \colon T \in \tria\} \subset \Ss^{0,1}_{\tria,0}$ and $\Theta_\tria=\{\theta_T \colon T \in \tria\} \subset H_{0,\gamma}^1(\Gamma)$
by $\sigma_{\tria,T}:=\sum_{\nu \in N^0_{\tria,T}} d_{\tria,\nu}^{-1} \phi_{\tria,\nu}$, $\theta_T:=\theta_{\new{\chi}^{-1}(T)} \circ \new{\chi}^{-1}$  on $T$ and $\theta_T := 0$ elsewhere.
 Thanks to our assumption of $\theta_{\new{\chi}^{-1}(T)} \geq 0$, it holds that $\langle \theta_{T},\xi_{T}\rangle_{L_2(\Gamma)} \eqsim \langle \theta_{\new{\chi}^{-1}(T)}, \xi_{\new{\chi}^{-1}(T)}\rangle_{L_2(\new{\chi}^{-1}(T))} \eqsim 
 \| \theta_{T}\|_{L_2(\Gamma)}\|\xi_{T}\|_{L_2(\Gamma)}$ (cf. \eqref{C1}).
 
 Now defining $\Psi_\tria:=\{\psi_{\tria,T} \colon T \in \tria\} $ and $\W_\tria:=\Span \Psi_\tria \subset \W$
 by
\be \label{expressionpsi}
\psi_{\tria,T} := \sigma_{\tria,T}+\frac{\langle \1-\sigma_{\tria,T},\xi_T\rangle_{L_2(\Gamma)}}{\langle \theta_T,\xi_T\rangle_{L_2(\Gamma)}} \theta_T-\sum_{T' \in \tria\setminus\{T\}} \frac{\langle \sigma_{\tria,T},\xi_{T'}\rangle_{L_2(\Gamma)}}{\langle \theta_{T'},\xi_{T'}\rangle_{L_2(\Gamma)}} \theta_{T'},
\ee
 and \new{$D_\tria\colon \V_\tria \rightarrow \W'_\tria$ by $(D_\tria v)(w):=(Dv)(w)=\langle v,w\rangle_{L_2(\Gamma)}$}, the analysis from Sect.~\ref{Sdomain_case} applies verbatim by only changing $\langle\,,\,\rangle_{L_2(\Omega)}$ into $\langle\,,\,\rangle_{L_2(\Gamma)}$. It yields that $\|D_\tria\|_{\cL(\V_\tria,\W_\tria')}\leq 1$, $\sup_{\tria \in \bbT}\|D^{-1}_\tria\|_{\cL(\W'_\tria,\V_\tria)} <\infty$, and $\bm{D}_\tria=\diag\{\langle \1,\xi_T\rangle_{L_2(\Gamma)}\colon T \in \tria\}$.
 
A \emph{hidden problem}, however, is that the computation of $\bm{D}_\tria$, and that of the scalar products \new{in \eqref{expressionpsi}}
 involve integrals over $\Gamma$ that generally have to be approximated using numerical quadrature. 
Recalling that, for $s>0$, the preconditioner $\bm{G}_\tria=\bm{D}_\tria^{-1} \bm{B}_\tria \bm{D}_\tria^{-\top}$ is \emph{not} a uniformly well-conditioned matrix, it is a priorily not clear which quadrature errors are allowable, in particular when $\tria$ is far from being quasi-uniform.
For this reason, in the following \S\ref{Smodified} we propose a slightly modified construction of $\W_\tria$ and $\bm{D}_\tria$ that does not require the evaluation of integrals over $\Gamma$.
\note{Phrase removed}

As a preparation, in the following lemma we present a non-standard inverse inequality on the family $(\V_\tria)_{\tria \in \bbT}$.
Proofs of this inequality for $d \leq 3$ can be found in \cite{56.5,75.64}.
It turns out that our construction of a `local' collection $\Psi_\tria \subset H^1_{0,\gamma}(\Omega)$ that is biorthogonal to $\Xi_\tria$ \note{Phrase removed}  allows for a very simple proof.

 \begin{lemma}[inverse inequality] \label{inverse}
With $h_\tria|_T:=h_T$, it holds that
 $$
 \|h_\tria v_\tria\|_{L_2(\Gamma)} \lesssim \|v_\tria\|_{H^1_{0,\gamma}(\Gamma)'} \quad (v_\tria \in \V_\tria).
 $$
 \end{lemma}
 
 \begin{proof} 
For $P_\tria:L_2(\new{\Gamma}) \rightarrow H^1_{0,\gamma}(\Gamma)$ defined by
 $$
 P_\tria u=\sum_{T \in \tria} \frac{\langle u,\xi_T \rangle_{L_2(\Gamma)}}{\langle \1,\xi_T\rangle_{L_2(\Gamma)}} \psi_{\tria,T}.
 $$
 we have \note{Phrase removed} 
 $\ran (\identity -P_\tria)=\V_\tria^{\perp_{L_2(\Gamma)}}$, and as follows from \eqref{10},
 $$
\| P_\tria u\|_{H^1(\Gamma)} \lesssim \|h_\tria^{-1} u\|_{L_2(\Gamma)} \quad(u \in L_2(\Gamma)).
$$
 The proof is completed by 
 \begin{align*}
 \|v_\tria\|_{H^1_{0,\gamma}(\Gamma)'}\!=\!\!\!\!\!\!\!\sup_{0 \neq w \in H^1_{0,\gamma}(\Gamma)}\!\!\!\! \frac{\langle v_\tria,w\rangle_{L_2(\Gamma)}}{\|w\|_{H^1(\Gamma)}}
\! \geq  \!\frac{\langle v_\tria,P_\tria h_\tria^2 v_\tria\rangle_{L_2(\Gamma)}}{\|P_\tria h_\tria^2 v_\tria\|_{H^1(\Gamma)}}
 \!\gtrsim \!\frac{\langle h_\tria v_\tria,h_\tria v_\tria\rangle_{L_2(\Gamma)}}{\|h_\tria v_\tria\|_{L_2(\Gamma)}}.\,\, \,\,\,\qedhere
 \end{align*}
 \end{proof}

\subsubsection{Modified construction for manifolds} \label{Smodified} \note{Phrase removed} Given $\tria \in \bbT$, on $L_2(\Gamma)$ we define an additional, `mesh-dependent' scalar product
 $$
 \langle u,v\rangle_{\tria}:=\sum_{T \in \tria} \frac{|T|}{|\new{\chi}^{-1}(T)|} \int_{\new{\chi}^{-1}(T)} u(\new{\chi}(x))v(\new{\chi}(x))dx.
 $$
 It is constructed from 
 $$
 \langle u,v\rangle_{L_2(\Gamma)}=\int_{\Omega} u(\new{\chi}(x))v(\new{\chi}(x)) |\partial \new{\chi}(x)|dx
 $$
 by replacing on each $\new{\chi}^{-1}(T)$, the Jacobian $|\partial \new{\chi}|$ by its average $\frac{|T|}{|\new{\chi}^{-1}(T)|}$ over $\new{\chi}^{-1}(T)$.%
 \footnote{It will be clear from the following that $\frac{|T|}{|\new{\chi}^{-1}(T)|}$ can be read as \emph{any} constant approximation to $|\partial \new{\chi}|$ on $L_\infty(\new{\chi}^{-1}(T))$-distance $\lesssim h_{\new{\chi}^{-1}(T)}$, for example $|\partial \new{\chi}(z)|$ \new{for} some $z \in\new{\chi}^{-1}(T)$.
 Then in the following, the volumes $|T|$ in the expression for $\bm{D}_\tria$ should be read as $|\new{\chi}^{-1}(T)| |\partial \new{\chi}(z)|$, with which also the computation of $|T|$ is avoided.}

 We now \emph{redefine} $\Psi_\tria:=\{\psi_{\tria,T} \colon T \in \tria\}$, $\W_\tria:=\Span \Psi_\tria \subset \W$
 by
$$
\psi_{\tria,T} := \sigma_{\tria,T}+\frac{\langle \1-\sigma_{\tria,T},\xi_T\rangle_{\tria}}{\langle \theta_T,\xi_T\rangle_{\tria}} \theta_T-\sum_{T' \in \tria\setminus\{T\}} \frac{\langle \sigma_{\tria,T},\xi_{T'}\rangle_{\tria}}{\langle \theta_{T'},\xi_{T'}\rangle_{\tria}} \theta_{T'},
$$
 and $D_\tria\colon\V_\tria \rightarrow \W'_\tria$ by $(D_\tria v_\tria)(w_\tria):=\langle v_\tria,w_\tria\rangle_\tria$.
Then, as in the domain case, \new{we get the explicit formulas}
$$
\bm{D}_\tria=\langle \Xi_\tria,\Psi_\tria\rangle_\tria =\diag\{\langle \1,\xi_T\rangle_{\tria}\colon T \in \tria\}=\diag\{|T|\colon T \in \tria\},
$$
and
\be \label{psi2}
 \psi_{\tria,T} = \sum_{\nu \in N^0_{\tria,T}} \hspace*{-0.4em}d_{\tria,\nu}^{-1} \phi_{\tria,\nu}+
\big( 1-{\textstyle \frac{1}{d+1}}\hspace*{-0.8em}\sum_{\nu \in N^0_{\tria,T}} d_{\tria,\nu}^{-1}\big) \theta_T
 - \hspace*{-0.4em}\sum_{T' \in \tria\setminus\{T\}} \hspace*{-0.6em}\big({\textstyle \frac{1}{d+1}}\hspace*{-1.3em}\sum_{\nu \in N^0_{\tria,T} \cap N^0_{\tria,T'}} \hspace*{-0.8em}d_{\tria,\nu}^{-1}\big) \theta_{T'}.
\ee
thus with coefficients that are \emph{independent} of $\new{\chi}$.

What remains is to prove the uniform boundedness of $\|D_\tria\|_{\cL(\V_\tria,\W_\tria')}$, and that of $\|D^{-1}_\tria\|_{\cL(\W'_\tria,\V_\tria)}$.
Because of the definition of $D_\tria$ in terms of the mesh-dependent scalar product, for doing so we cannot simply rely on the \new{`Fortin criterion' from} Proposition~\ref{fortin}.

\begin{lemma} \label{lemmie}
    It holds that $\sup_{\tria \in \bbT}\|D_\tria\|_{\cL(\V_\tria,\W_\tria')}<\infty$.
\end{lemma}
\begin{proof}
\new{If} $s=0$, i.e., when $\W=L_2(\Gamma) \simeq L_2(\Gamma)'=\V$, \new{then} the uniform boundedness of $\|D_\tria\|_{\cL(\V_\tria,\W'_\tria)}$ follows directly from $\langle\cdot,\cdot\rangle_\tria \eqsim \|\cdot\|_{L_2(\Gamma)}^2$.

 By an interpolation argument, in the following it suffices to consider the case $s=1$, i.e., $\W=H^1_{0,\gamma}(\Gamma)$ and $\V=H^1_{0,\gamma}(\Gamma)'$. 
By definition of $\langle \,,\,\rangle_{\tria}$, it holds that
 \be \label{closeness}
|\langle v,u\rangle_\tria-\langle v,u\rangle_{L_2(\Gamma)}| \lesssim \|h_\tria v\|_{L_2(\Gamma)} \|u\|_{L_2(\Gamma)} \quad (v,u \in L_2(\Gamma)).
\ee
  By writing $(D_\tria v_\tria)(w_\tria)=\langle v_\tria,w_\tria \rangle_{L_2(\Gamma)}+\langle v_\tria,w_\tria\rangle_\tria-\langle v_\tria,w_\tria\rangle_{L_2(\Gamma)}$, the uniform boundedness of  $\|D_\tria\|_{\cL(\V_\tria,\W'_\tria)}$ (for $s=1$) now follows by combining \eqref{closeness} and Lemma~\ref{inverse}.
\end{proof}

The $\langle\,,\,\rangle_\tria$-biorthogonal projector $\check{P}_\tria:L_2(\Omega) \rightarrow H^1_{0,\gamma}(\Omega)$ with $\ran \check{P}_\tria=\W_\tria$ and $\ran (\identity -\check{P}_\tria)=\V_\tria^{\perp_{\langle\,,\,\rangle_\tria}}$ exists and  is given by
 $\check{P}_\tria u=\sum_{T \in \tria} |T|^{-1} \langle u,\xi_T \rangle_{\tria} \psi_{\tria,T}$.
 Since $\langle \,,\,\rangle_{\tria}$ gives rise to a norm that is uniformly equivalent to $\|\,\|_{L_2(\Gamma)}$, the proof of Theorem~\ref{theorem1} again applies, and shows that
 $$
 \sup_{\tria \in \bbT}\|\check{P}_\tria\|_{\cL(L_2(\Gamma),L_2(\Gamma))}<\infty, \quad\sup_{\tria \in \bbT}\|\check{P}_\tria\|_{\cL(H^1_{0,\gamma}(\Gamma),H^1_{0,\gamma}(\Gamma))}<\infty,
 $$
  as well as 
\be \label{25}
\|\check{P}_\tria u\|_{H^1(\Gamma)} \lesssim \|h_\tria^{-1} u\|_{L_2(\Gamma)} \quad(u \in L_2(\Gamma)).
\ee

\new{These properties of $\check{P}_\tria$ will be the key to prove the uniform boundedness of $\|D^{-1}_\tria\|_{\cL(\W'_\tria,\V_\tria)}$.} Indeed, \new{for $s=0$} uniform boundedness of $\|D^{-1}_\tria\|_{\cL(\V_\tria,\W'_\tria)}$  follows from
$$
(D_\tria v_\tria)(\check{P}_\tria v_\tria)=\langle v_\tria,v_\tria\rangle_\tria \eqsim \|v_\tria\|_{L_2(\Gamma)}^2 \gtrsim \|v_\tria\|_{L_2(\Gamma)} \|\check{P}_\tria v_\tria\|_{L_2(\Gamma)}.
$$
To conclude, by an interpolation argument, uniform boundedness of $\|D^{-1}_\tria\|_{\cL(\V_\tria,\W'_\tria)}$ for any $s \in [0,1]$, it is sufficient to verify the case $s=1$, which \new{will} be done \new{using} the following modified inverse inequality.
 
 \begin{lemma} \label{lem4} It holds that
 $$
 \|h_\tria v_\tria\|_{L_2(\Gamma)} \lesssim \sup_{0 \neq w \in H^1_{0,\gamma}(\Gamma)} \frac{\langle v_\tria,w\rangle_\tria}{\|w\|_{H^1(\Gamma)}} \quad (v_\tria \in \V_\tria).
 $$
 \end{lemma}
 
 \begin{proof} Similar to proof of Lemma~\ref{inverse}, using \eqref{25} for $v_\tria \in \V_\tria$ we estimate
 \begin{align*}
\sup_{0 \neq w \in H^1_{0,\gamma}(\Gamma)}
 \frac{\langle v_\tria,w\rangle_{\tria}}{\|w\|_{H^1(\Gamma)}}
\! \geq  \!\frac{\langle v_\tria,\check{P}_\tria h_\tria^2 v_\tria\rangle_{\tria}}{\|\check{P}_\tria h_\tria^2 v_\tria\|_{H^1(\Gamma)}}
 \!\gtrsim \!\frac{\langle h_\tria v_\tria,h_\tria v_\tria\rangle_{\tria}}{\|h_\tria v_\tria\|_{L_2(\Gamma)}} \eqsim \|h_\tria v_\tria\|_{L_2(\Gamma)}.\hspace*{1em}\qedhere
 \end{align*}
 \end{proof}

 \begin{corollary} \label{corol1} It holds that
 $$
 \|v_\tria\|_{H^1_{0,\gamma}(\Gamma)'} \eqsim \sup_{0 \neq w_\tria \in \W_\tria} \frac{\langle v_\tria,w_\tria \rangle_{\tria}}{\|w_\tria\|_{H^1(\Gamma)}}  \quad (v_\tria \in \V_\tria),
$$
{\rm(} with `$\lesssim$' being the statement  $\sup_{\tria \in \bbT} \|D^{-1}_\tria\|_{\cL(\W'_\tria,\V_\tria)}<\infty$ for $s=1${\rm)}.
\end{corollary}

\begin{proof}
The inequality `$\gtrsim$' is the statement of Lemma~\ref{lemmie} for $s=1$.

To prove the other direction, for $v \in L_2(\Gamma)$, \eqref{closeness} shows that
 $$
 \new{\Big|}\|v\|_{H^1_{0,\gamma}(\Gamma)'}-\sup_{0 \neq w \in H^1_{0,\gamma}(\Gamma)} \frac{\langle v,w\rangle_{\tria}}{\|w\|_{H^1(\Gamma)}}
 \new{\Big|} \lesssim \|h_\tria v\|_{L_2(\Gamma)},
  $$
  Taking $v=v_\tria \in \V_\tria$, from Lemma~\ref{lem4} we conclude that
   \begin{align*}
 \|v_\tria\|_{H^1_{0,\gamma}(\Gamma)'} & \lesssim \sup_{0 \neq w \in H^1_{0,\gamma}(\Gamma)} \frac{\langle v_\tria,w\rangle_{\tria}}{\|w\|_{H^1(\Gamma)}}=
\sup_{0 \neq w \in H^1_{0,\gamma}(\Gamma)}  \frac{\langle v_\tria,\check{P}_\tria w\rangle_{\tria}}{\|w\|_{H^1(\Gamma)}} \\
 & \hspace*{-3em}\leq \|\check{P}_\tria\|_{\cL(H^1_{0,\gamma}(\Gamma),H^1_{0,\gamma}(\Gamma))} \sup_{0 \neq w_\tria \in \W_\tria} \frac{\langle v_\tria,w_\tria \rangle_{\tria}}{\|w_\tria\|_{H^1(\Gamma)}}
\lesssim \sup_{0 \neq w_\tria \in \W_\tria} \frac{\langle v_\tria,w_\tria \rangle_{\tria}}{\|w_\tria\|_{H^1(\Gamma)}}
  \end{align*}
  by $\sup_{\tria \in \bbT}\|\check{P}_\tria\|_{\cL(H^1_{0,\gamma}(\Gamma),H^1_{0,\gamma}(\Gamma))}<\infty$.
\end{proof}
 
 \subsection{Construction of $B_\tria \in \cLis(\W_\tria, \W'_\tria)$.} \label{SB}
Having established, in both domain and manifold case, $\sup_{T \in \bbT} \max\big(\|D_\tria\|_{\cL(\V_\tria,\W_\tria')}, \|D^{-1}_\tria\|_{\cL(\W_\tria',\V_\tria)}\big)<\infty$,
for the construction of uniform preconditioners it remains to find $B_\tria \in \cLis(\W_\tria, \W'_\tria)$ with 
$\sup_{T \in \bbT} \max\big(\|B_\tria\|_{\cL(\W_\tria,\W_\tria')}, \|\Re(B_\tria)^{-1}\|_{\cL(\W_\tria',\W_\tria)}\big)<\infty$.

\new{We add the following two assumptions on the collection $\Theta_\tria$ of `bubbles' and their span $\B_\tria:=\Span \Theta_\tria$. For $k \in \{0,1\}$ it holds that
\begin{align} \label{14}
\|\sum_{T \in \tria} c_T \theta_T\|^2_{H^k(\Omega)} &\eqsim \sum_{T \in \tria} h_T^{-2k} |c_T|^2 \| \theta_T\|_{L_2(\Omega)}^2 ,\quad ((c_T)_{T \in \tria} \subset \R),
\intertext{and} \label{stab}
\|u+v\|_{H^k(\Omega)}^2 &\gtrsim \|u\|_{H^k(\Omega)}^2+\|v\|_{H^k(\Omega)}^2 \quad(u \in \Ss^{0,1}_{\tria,0},\,v \in \B_\tria).
\end{align}}%
(Here and in the following, $\Omega$ should be read as $\Gamma$ in the manifold case).
\new{Both properties are easily verified by a standard homogeneity argument for both our earlier specifications of possible $\Theta_\tria$. From \eqref{stab} it follows that $\Ss^{0,1}_{\tria,0} \cap \B_\tria=\{0\}$. Let $I_\tria^\Ss$ be the linear projector defined on $\Ss^{0,1} \oplus \B_\tria$ by $\ran I_\tria^\Ss=\Ss^{0,1}_{\tria,0}$ and $\ran I_\tria^\B =\B_\tria$, where $I_\tria^\B:=\identity-I_\tria^\Ss$.}

  Below we give a construction of suitable $B_\tria$ that is \emph{independent} of the particular bubbles $\Theta_\tria$ being chosen.
 Like $\W_\tria$, we equip $\Ss^{0,1}_{\tria,0}$, $\B_\tria$, and $\Ss^{0,1}_{\tria,0}\oplus \B_\tria$ with $\|\,\|_\W$.
 
 \begin{proposition} \label{prop3}
 Given $B_\tria^{\Ss} \in \cLis(\Ss^{0,1}_{\tria,0},(\Ss^{0,1}_{\tria,0})')$ and $B_\tria^{\B} \in \cLis(\B_\tria,\B_\tria')$, let $B^{\new{\Ss\oplus\B}}_\tria \colon \Ss^{0,1}_{\tria,0}\oplus \B_\tria \rightarrow (\Ss^{0,1}_{\tria,0} \oplus \B_\tria)'$ be defined by
 $$
 (B^{\Ss\oplus\B}_\tria \new{w})(\new{\tilde w})\new{:=}(B^{\Ss}_\tria \new{I_\tria^\Ss w})(\new{I_\tria^\Ss \tilde w})+(B^\B_\tria \new{I_\tria^\B w})(\new{I_\tria^\B \tilde w}).
 $$
 Then \new{thanks to \eqref{stab}, one has} $B^{\Ss\oplus\B}_\tria  \in \cLis(\Ss^{0,1}_{\tria,0}\oplus \B_\tria ,(\Ss^{0,1}_{\tria,0}\oplus \B_\tria )')$, and
 \begin{align*}
 & \|\Re(B^{\Ss\oplus\B}_\tria)^{-1}\|_{\cL((\Ss^{0,1}_{\tria,0}\oplus \B_\tria)',\Ss^{0,1}_{\tria,0}\oplus \B_\tria)} \\
 &\hspace*{5em} \leq 2 \max(\|\Re(B_\tria^{\Ss} )^{-1}\|_{\cL((\Ss^{0,1}_{\tria,0})',\Ss^{0,1}_{\tria,0})},\|\Re(B_\tria^\B)^{-1}\|_{\cL(\B'_\tria,\B_\tria)}),\\
& \|B^{\Ss\oplus\B}_\tria\|_{\cL(\Ss^{0,1}_{\tria,0}\oplus \B_\tria,(\Ss^{0,1}_{\tria,0}\oplus \B_\tria)')}   \lesssim  \max(\|B_\tria^{\Ss} \|_{\cL(\Ss^{0,1}_{\tria,0},(\Ss^{0,1}_{\tria,0})')},\|B_\tria^\B\|_{\cL(\B_\tria,\B_\tria')}).
\end{align*}
 \end{proposition}
  
 \begin{proof}  One has
 \begin{align*}
  |(B^{\Ss\oplus\B}_\tria \new{w})(\new{w})|\geq & \min(\|\Re(B_\tria^{\Ss} )^{-1}\|^{-1}_{\cL((\Ss^{0,1}_{\tria,0})',\Ss^{0,1}_{\tria,0})},\|\Re(B_\tria^\B)^{-1}\|^{-1}_{\cL(\B'_\tria,\B_\tria)})\\
  & \times( \|\new{I_\tria^\Ss w}\|_\W^2+\|\new{I_\tria^\B w}\|_\W^2),
  \end{align*}
and 
   \begin{align*}
 |(B^{\Ss\oplus\B}_\tria w)(\tilde w)|\leq &\max(\|B_\tria^{\Ss} \|_{\cL(\Ss^{0,1}_{\tria,0},(\Ss^{0,1}_{\tria,0})')},\|B_\tria^\B\|_{\cL(\B_\tria,\B_\tria')}) \\
 & \times \sqrt{\|\new{I_\tria^\Ss w}\|_\W^2+\|\new{I_\tria^\B w}\|_\W^2}\sqrt{\|\new{I_\tria^\Ss \tilde w}\|_\W^2+\|\new{I_\tria^\B \tilde w}\|_\W^2}.
 \end{align*}
\new{ From the triangle inequality and \eqref{stab}, one has $\frac12\|w\|_\W^2 \leq \|I_\tria^\Ss w\|^2_\W+\|I_\tria^\B w\|^2_\W
\lesssim  \|w\|_\W^2$, which 
 completes the proof.}
 \end{proof}
 
\new{By equipping $\W_\tria$, $\Ss^{0,1}_{\tria,0}$ and $\B_\tria$ by $\Psi_\tria$, $\Phi_\tria$, and $\Theta_\tria$, respectively, the applications of $I_\tria^\Ss|_{\W_\tria}$ and $I_\tria^\B|_{\W_\tria}$ can easily determined in linear complexity. Therefore} 
  a suitable definition of $B_\tria\colon\W_\tria\rightarrow \W_\tria'$ is given by $(B_\tria w)(\tilde w)\new{:=}(B^{\Ss\oplus\B}_\tria w)(\tilde w)$. Clearly,
\begin{align*}
\|B_\tria\|_{\cL(\W_\tria,\W_\tria')} &\leq \|B^{\Ss\oplus\B}_\tria\|_{\cL(\Ss^{0,1}_{\tria,0}\oplus \B_\tria,(\Ss^{0,1}_{\tria,0}\oplus \B_\tria)')},\\
\|\Re(B_\tria)^{-1}\|_{\cL(\W_\tria',\W_\tria)} &\leq  \|\Re(B^{\Ss\oplus\B}_\tria)^{-1}\|_{\cL((\Ss^{0,1}_{\tria,0}\oplus \B_\tria)',\Ss^{0,1}_{\tria,0}\oplus \B_\tria)}.
\end{align*}

 An obvious choice for $B_\tria^{\B} \in \cLis(\B_\tria,(\B_\tria)')$ such that
 $$
\max\big( \sup_{\tria\in \bbT}  \|B_\tria^\B\|_{\cL(\B_\tria,\B_\tria')}, \|\Re(B_\tria^\B)^{-1}\|_{\cL(\B'_\tria,\B_\tria)}\big)<\infty,
 $$
 is, in view of \eqref{14} and \new{$\|\theta_T\|_{L_2(\Omega)} \stackrel{\eqref{C1}}{\eqsim} \frac{\langle \theta_T,\xi_T\rangle_{L_2(\Omega)}}{\|\xi_T\|_{L_2(\Omega)}}
 \stackrel{\eqref{scale}}{=} |T| \|\xi_T\|_{L_2(\Omega)}^{-1}= h^{d/2}$}, given by 
\be \label{17}
\big(B_\tria^{\B} \sum_{T \in \tria} c_T \theta_T\big)\big(\sum_{T \in \tria} d_T \theta_T\big):=\beta \sum_{T\in \tria} h_T^{d-2s} c_T d_T.
\ee
for some constant $\beta$.

Possible choices for $B_\tria^{\Ss}  \in \cLis(\Ss^{0,1}_{\tria,0},(\Ss^{0,1}_{\tria,0})')$ with 
$$
\sup_{\tria\in \bbT} \big( \|B_\tria^{\Ss} \|_{\cL(\Ss^{0,1}_{\tria,0},(\Ss^{0,1}_{\tria,0})')}, \|\Re(B_\tria^{\Ss} )^{-1}\|_{\cL((\Ss^{0,1}_{\tria,0})',\Ss^{0,1}_{\tria,0})}\big)<\infty
$$
include 
$ (B_\tria^{\Ss}  u)(v):=(B u)(v)$ ($u,v \in \Ss^{0,1}_{\tria,0}$) for some $B \in \Lis_c(\W,\W')$.

For $d \in \{2,3\}$ and $\W=H^{\frac{1}{2}}_{00}(\Gamma):=[L_2(\Gamma),H^1_0(\Gamma)]_{\frac{1}{2},2}$,  for \new{this} $B$ one may take the hypersingular integral operator, whereas for $\partial\Gamma \neq \emptyset$, and $\W=H^{\frac{1}{2}}(\Gamma)=[L_2(\Gamma),H^1(\Gamma)]_{\frac{1}{2},2}$
the recently introduced modified hypersingular integral operator can be applied (see \cite{138.28}).
(Note that $H^1_0(\Gamma)=H^1(\Gamma)$ when $\partial\Gamma = \emptyset$.)

For a family of quasi-uniform partitions generated by a repeated application of \emph{uniform refinements} starting from some given initial partition, a computationally attractive alternative \new{choice for $B_\tria^{\Ss}$} is provided by the
multi-level \new{operator} from \cite{34.8}, whose application can be performed in \emph{linear complexity}.
\new{In \cite{249.98}, such operators will be discussed that also apply on locally refined meshes.}

\new{For $\W=H^1_{0,\gamma}(\Omega)$, i.e., when $A$ is an operator of order $-2$ (cf. \cite{75.067}), one obviously takes $(B_\tria^{\Ss}u)(v)=\int_\Omega \nabla u \cdot \nabla v\,dx$, or $(B_\tria^{\Ss}u)(v)=\int_\Omega \nabla u \cdot \nabla v\,dx+\int_\Omega u  v\,dx$ when $\meas(\gamma)=0$, whose application can be performed in linear complexity.} 

\subsection{Implementation} \label{impl} 
 For both the domain case and the construction in the manifold case in \S\ref{Smodified}, the matrix representation $\bm{G}_\tria=\cF_{\Xi_\tria}^{-1} G_\tria (\cF'_{\Xi_\tria})^{-1}$ of our preconditioner $G_\tria$ reads as \framebox{$\bm{G}_\tria=\bm{D}_\tria^{-1} \bm{B}_\tria \bm{D}_\tria^{-\top}$} with 
 $$
 \framebox{$\bm{D}_\tria=\diag\{|T|\colon T \in \tria\}$},
 $$
 and
  \begin{align*}
\framebox{$\bm{B}_\tria$}& := \cF'_{\Psi_\tria} B_\tria \cF_{\Psi_\tria}\\
&= \cF'_{\Psi_\tria}
\big(\new{({I^\Ss_\tria}|_{\W_\tria})'} B_\tria^{\Ss}  \new{{I^\Ss_\tria}|_{\W_\tria}}+
\new{ {(I^\B_\tria}|_{\W_\tria})'} B_\tria^{\B} \new{I^\B_\tria|_{\W_\tria}} \big)\cF_{\Psi_\tria}\\
& =\framebox{$\bm{p}_\tria^\top \bm{B}_\tria^{\Ss} \bm{p}_\tria+\bm{q}_\tria^\top \bm{B}^{\B}_\tria \bm{q}_\tria$},
\end{align*}
where
\begin{alignat*}{2}
\bm{B}_\tria^{\Ss}&:=\cF_{\Phi_\tria}' B_\tria^{\Ss}  \cF_{\Phi_\tria},& \quad \bm{p}_\tria&:=\cF^{-1}_{\Phi_\tria}  \new{{I^\Ss_\tria}|_{\W_\tria}}\cF_{\Psi_\tria},\\
\bm{B}^{\B}_\tria&:=\cF_{\Theta_\tria}'B_\tria^{\B} \cF_{\Theta_\tria},&\quad \bm{q}_\tria&:=\cF^{-1}_{\Theta_\tria}\new{{I^\B_\tria}|_{\W_\tria}} \cF_{\Psi_\tria}.
  \end{alignat*}
By substituting the definition of $B_\tria^{\B}$ from \eqref{17},  the definition of the basis $\Psi_\tria=\{\psi_{\tria,T}\}_{T \in \tria}$ for $\W_\tria$ from \eqref{psi1} and \eqref{psi2},
 and that of the bases $\Phi_\tria=\{\phi_{\tria,\nu}\}_{\nu \in N^0_\tria}$
 and 
 $\Theta_\tria=\{\theta_T\}_{T \in \tria}$ for $\Ss^{0,1}_{\tria,0}$ and $\B_\tria$, respectively, we find that 
\begin{align*}
&\framebox{$\bm{B}^{\B}_\tria=  \beta \bm{D}_\tria^{1-\frac{2s}{d}}$},
\quad
\framebox{$(\bm{p}_\tria)_{\nu T}=\left\{\begin{array}{cl} d_{\tria,\nu}^{-1} & \text{if } \nu \in N^0_{\tria,T},\\ 0 & \text{if }\nu \not\in N^0_{\tria,T},\end{array}\right.$}\\
& \framebox{$\displaystyle (\bm{q}_\tria)_{T' T}=\delta_{T' T}- {\textstyle \frac{1}{d+1}}\hspace*{-1em}\sum_{\nu \in N^0_{\tria,T} \cap N^0_{\tria,T'}} d_{\tria,\nu}^{-1}$},
\end{align*}
whereas \framebox{$\bm{B}_\tria^{\Ss}$} depends on $B_\tria^{\Ss}  \in \cLis(\Ss^{0,1}_{\tria,0},(\Ss^{0,1}_{\tria,0})')$ being chosen.
\new{The cost of the application of $\bm{G}_\tria$ is the cost of the application of $\bm{B}_\tria^{\Ss}$  plus cost that scales linearly in $\# \tria$.}
\note{Less relevant remark removed}

\section{Preconditioning an operator of negative order discretized by {\fontshape{scit}\selectfont continuous} piecewise \new{linears}.} \label{Slinears}
Let a bounded polytopal domain $\Omega \subset \R^d$, $\gamma\subset \partial\Omega$, $s \in [0,1]$, 
 $\W:=[L_2(\Omega),H^1_{0,\gamma}(\Omega)]_{s,2}$, $\V:=\W'$, 
$D \in \Lis(\V,\W')$, $(\tria)_{\tria \in \bbT}$, $N^0_\tria$, $d_{\tria,\nu}$, $N_T$, and $N^0_{\tria,T}$ be all as in Sect.~\ref{Sdomain_case}. 
 In addition, for $\tria \in \bbT$ let $N_\tria$ be the set of \emph{all} vertices of $\tria$, \new{so including those on a possibly non-empty $\gamma$}, and
for $\nu \in N_\tria$ let $\omega_\tria(\nu):=\cup_{\{T \in \tria\colon \nu \in  N_T\}} \overline{T}$.

 We take
 $$
\framebox{$\V_\tria=\Ss^{0,1}_\tria:=\{u \in H^1(\Omega)\colon u|_T \in \cP_1 \,(T \in \tria)\} \subset \V$},
$$
and, as in Sect.~\ref{Sdomain_case},
$$
\Ss^{0,1}_{\tria,0}:=\{u \in H^1_{0,\gamma}(\Omega)\colon u|_T \in \cP_1 \,(T \in \tria)\},
 $$
 equipped with nodal bases $\Xi_\tria=\{\xi_{\tria,\nu}\colon \nu \in N_\tria\}$ and $\Phi_\tria=\{\phi_{\tria,\nu}\colon \nu \in N^0_\tria\}$, respectively,
defined by
 $$
  \xi_{\tria,\nu}(\nu')=\delta_{\nu,\nu'} \quad (\nu,\nu'\in N_\tria),
$$
and $\phi_{\tria,\nu}=\xi_{\tria,\nu}$ for $\nu \in N^0_\tria$.

Analogously to the case of \emph{discontinuous} piecewise \new{constant} trial spaces in $\V$ studied in Sect.~\ref{Sdomain_case}, using the framework of operator preconditioning outlined in Sect.~\ref{Soper} we are going to construct a family of preconditioners $G_\tria \in \cLis(\V_\tria',\V_\tria)$ of type $D_\tria^{-1} B_\tria (D_\tria')^{-1}$ with uniformly bounded 
$\|G_\tria\|_{\cL(\V_\tria',\V_\tria)}$ and $\|\Re(G_\tria)^{-1}\|_{\cL(\V_\tria,\V_\tria')}$.

\new{The roles played in Sect.~\ref{Sdomain_case} by $|T|$ ($= |\supp \xi_T|$) and $h_T=|T|^{1/d}$, are in this section going to be played by $|\omega_\tria(\nu)|$ ($= |\supp \xi_{\tria,\nu}|$) and $h_{\tria,\nu}:=|\omega_\tria(\nu)|^{1/d}$.}
 
\subsection{Construction of $\W_\tria$ and $\bm{D}_\tria$} \label{SlinearsWD} To construct a collection $\Psi_\tria=\{\psi_{\tria,\nu}\colon \nu \in N_\tria\} \subset H^1_{0,\gamma}(\Omega)$ that is biorthogonal to $\Xi_\tria$, \new{consists of locally supported functions}, and for which
$$
\framebox{$\W_\tria:=\Span \Psi_\tria \subset \W$}
$$
has an `approximation property', as in Sect.~\ref{Sdomain_case} we need two collections $\Sigma_\tria \subset \Ss^{0,1}_{\tria,0}$ and $\Theta_\tria \subset H^1_{0,\gamma}(\Omega)$ \new{of locally supported functions with $\#\Sigma_\tria=\# \Theta_\tria=\# \Xi_\tria$}, where $\Theta_\tria$ is biorthogonal to $\Xi_\tria$, and $\Sigma_\tria$ has an `approximation property'.

We define $\Sigma_\tria=\{\sigma_{\tria,\nu}\colon \nu \in N_\tria\}$ by $\sigma_{\tria,\nu}:=\phi_{\tria,\nu}$ when $\nu \in N^0_\tria$, and $\sigma_{\tria,\nu}:=0$ when $\nu \in N_\tria \setminus  N^0_\tria$. Then, obviously, $\sum_{\nu \in N_\tria} \sigma_{\tria,\nu}$ equals $\1$ on $\Omega \setminus \cup_{\{T \in \tria\colon \overline{T} \cap \gamma \neq \emptyset\}} \overline{T}$.

For constructing $\Theta_\tria$,  on a reference $d$-simplex $\hat{T}$ for $\eps>0$ we consider a smooth $\eta_\eps \in [0,1]$, symmetric in the barycentric coordinates, with $\eta_\eps(x)=0$ when $d(x,\partial\hat{T})<\eps$, and $\eta_\eps(x)=1$ when $d(x,\partial\hat{T})>2\eps$.
Then for some fixed $\eps>0$ small enough, it holds that
\be \label{localinfsup}
\inf_{0 \neq p \in \cP_1(\hat{T})} \sup_{0 \neq q \in \cP_1(\hat{T})} \frac{\langle p,\eta_\eps q\rangle_{L_2(\hat{T})}}{\|p\|_{L_2(\hat{T})}\|\eta_\eps q\|_{L_2(\hat{T})}}>0,
\ee
meaning that the biorthogonal projector $P_\eps \in \cL(L_2(\hat{T}),L_2(\hat{T}))$ with $\ran P_\eps=\eta_\eps \cP_1(\hat{T})$ and $\ran(\identity-P_\eps)=\cP_1(\hat{T})^{\perp_{L_2(\hat{T})}}$ exists.
Consequently, with $\Phi_{\hat{T}}=\{\phi_{\hat{T},\nu}\colon \nu \in N_{\hat{T}}\}$ being the nodal basis for $\cP_1(\hat{T})$, we have that 
$$
\{\tilde{\phi}_{\hat{T},\eps,\nu}\colon \nu \in N_{\hat{T}}\}:=\langle \Phi_{\hat{T}},\Phi_{\hat{T}}\rangle_{L_2(\hat{T})}^{-1} P_\eps \Phi_{\hat{T}} \subset H^1_0(\hat{T})
$$
 is $L_2(\hat{T})$-biorthonormal to $\{\phi_{\hat{T},\nu}\colon \nu \in N_{\hat{T}}\}$.

Now for $T \in \tria$, let $F_{\hat{T},T}:T \rightarrow \hat{T}$ be an affine bijection. Then $\{\tilde{\phi}_{T,\eps,\nu}\colon \nu \in N_{T}\}$ defined by
\be \label{phitilde}
\tilde{\phi}_{T,\eps,\nu}:={\textstyle \frac{|\hat{T}|}{|T|}} \tilde{\phi}_{\hat{T},\eps,F_{\hat{T},T}(\nu)}
\ee
is $L_2(T)$-biorthonormal to the nodal basis for $P_1(T)$.

By selecting for $\nu \in N_\tria$, a $T(\nu) \in \tria$ with $\nu \in N_T$, and defining $\Theta_\tria=\{\theta_{\tria,\nu} \colon \nu \in N_\tria\} \subset H^1_{0,\gamma}(\Omega)$
by 
$$
\theta_{\tria,\nu} := |\omega_\tria(\nu)| \tilde{\phi}_{T(\nu),\eps,\nu},
$$
where the specific scaling is chosen for convenience, 
we have for $\nu,\, \nu'\in N_\tria$,
\be \label{props}
\begin{split}
& \delta_{\nu \nu'} |\omega_\tria(\nu)|=\langle \theta_{\tria,\nu},\xi_{\tria,\nu'}\rangle_{L_2(\Omega)} \eqsim \delta_{\nu \nu'} \|\theta_{\tria,\nu}\|_{L_2(\Omega)} \|\xi_{\tria,\nu'}\|_{L_2(\Omega)},\\
&\supp \theta_{\tria,\nu} \subset \overline{T(\nu)},\,\,
|\theta_{\tria,\nu}|_{H^1(\Omega)}  \lesssim \new{h_{\tria,\nu}^{-1}} \|\theta_{\tria,\nu}\|_{L_2(\Omega)},
\end{split}
\ee
i.e., properties analogous to \eqref{C1}, \new{\eqref{e1}(c), and \eqref{e3}}.

Since furthermore $\Sigma_\tria$ and $\Xi_\tria$ satisfy properties analogous to \new{\eqref{e1}(a,b), \eqref{e2} and \eqref{e4}}, defining similarly to \eqref{defpsi}
\be \label{deflinpsi}
\framebox{$\begin{split}
&\psi_{\tria,\nu} := \sigma_{\tria,\nu}+\frac{\langle \1-\sigma_{\tria,\nu},\xi_{\tria,\nu}\rangle_{L_2(\Omega)}}{\langle \theta_{\tria,\nu},\xi_{\tria,\nu}\rangle_{L_2(\Omega)}} \theta_{\tria,\nu}-\!\!\sum_{\nu' \in N_\tria\setminus\{\nu\}} \!\!\frac{\langle \sigma_{\tria,\nu},\xi_{\tria,\nu'}\rangle_{L_2(\Omega)}}{\langle \theta_{\tria,\nu'},\xi_{\tria,\nu'}\rangle_{L_2(\Omega)}} \theta_{\tria,\nu'}\\
&= \left\{
\begin{array}{@{}ll@{}}
{\textstyle \frac{\theta_{\tria,\nu}}{d+1}} & \nu \in N_\tria \setminus N^0_\tria,\\
 \phi_{\tria,\nu}+\frac{d}{(d+2)(d+1)}\theta_{\tria,\nu}-\!\!{\displaystyle\sum_{\nu' \in N_\tria\setminus\{\nu\}}} \!\!\frac{|\omega_\tria(\nu)\cap\omega_\tria(\nu')|}{(d+2)(d+1)|\omega_\tria(\nu')|}\theta_{\tria,\nu'}& \nu \in  N^0_\tria,
\end{array}
\right.
\end{split}$}\hspace*{-2.7ex}
\ee
\new{we infer that $\sum_{\nu \in N_\tria}\psi_{\tria,\nu}$ equals $\1$ possibly except on a strip along the Dirichlet boundary, and similarly to Theorem~\ref{theorem1}, that
the biorthogonal projector
\be \label{biorthprojlin}
P_\tria\colon u \mapsto \sum_{\nu \in N_\tria} \frac{\langle u,\xi_{\tria,\nu}\rangle_{L_2(\Omega)}}{\langle \1,\xi_{\tria,\nu}\rangle_{L_2(\Omega)}}\psi_{\tria,\nu},
\ee
satisfies $\sup_{\tria \in \bbT} \|P_\tria\|_{\cL(\W,\W)}<\infty$.} 
With $(D_\tria v)(w):=(D v)(w)$ ($(v,w) \in \V_\tria\times \W_\tria$), we have  $\|D_\tria\|_{\cL(\V_\tria,\W_\tria')} \leq 1$ and $\sup_{\tria \in \bbT} \|D^{-1}_\tria\|_{\cL(\W'_\tria,\V_\tria)}<\infty$, and
$$
\framebox{$\bm{D}_\tria= \cF_{\Psi_\tria}' D_\tria \cF_{\Xi_\tria}=\diag\{\langle \1,\xi_{\tria,\nu}\rangle_{L_2(\Omega)}\colon \nu \in N_\tria\}=\diag\Big\{{\textstyle \frac{1}{d+1}} |\omega_\tria(\nu)| \colon \nu \in N_\tria\Big\}$}.
$$

\subsection{Construction of $B_\tria \in \cLis(\W_\tria,\W_\tria')$.} \label{Bmani}
Since $\Theta_\tria$ additionally satisfies, for $k \in \{0,1\}$, 
\begin{align*}
&\big\|\sum_{\nu \in N_\tria} c_\nu \theta_{\tria,\nu}\big\|^2_{H^k(\Omega)} \eqsim \sum_{\nu \in N_\tria} \new{h_{\tria,\nu}^{-k}} \|\theta_{\tria,\nu}\|_{L_2(\Omega)}^2 |c_\nu|^2, 
\intertext{where \new{$\|\theta_{\tria,\nu}\|_{L_2(\Omega)}\stackrel{\eqref{props}}{=}|\omega_\tria(\nu)|\|\xi_{\tria,\nu}\|_{L_2(\Omega)}^{-1}\eqsim |\omega_\tria(\nu)|^{\frac12}$}, and}
&\|u+v\|_{H^k(\Omega)}^2 \gtrsim \|u\|_{H^k(\Omega)}^2+\|v\|_{H^k(\Omega)}^2 \quad(u \in \Ss^{0,1}_{\tria,0},\,v \in \B_\tria:=\Span \Theta_\tria).
\end{align*}
(cf. \eqref{14}-\eqref{stab}), we construct $B_\tria$ analogously as in \S\ref{SB}:
Assuming that we have a $B_\tria^{\Ss} \in \cLis(\Ss^{0,1}_{\tria,0},(\Ss^{0,1}_{\tria,0})')$ available
with $\sup_{\tria\in \bbT}  \|B_\tria^{\Ss} \|_{\cL(\Ss^{0,1}_{\tria,0},(\Ss^{0,1}_{\tria,0})')}<\infty$ and $\sup_{\tria\in \bbT}  \|\Re(B_\tria^{\Ss} )^{-1}\|_{\cL((\Ss^{0,1}_{\tria,0})',\Ss^{0,1}_{\tria,0})}<\infty$,
for some constant $\beta>0$ we take
$$
\big(B_\tria^{\B} \sum_{\nu \in N_\tria} c_\nu \theta_{\tria,\nu}\big)\big(\sum_{\nu \in   N_\tria} d_\nu \theta_{\tria,\nu}\big):=\beta \sum_{\nu \in N_\tria} |\omega_\tria(\nu)|^{1-\frac{2s}{d}}  c_\nu d_\nu,
$$
and
$$
B_\tria:=  \new{(I_\tria^\Ss|_{\W_\tria})'} B_\tria^{\Ss}  \new{I_\tria^\Ss|_{\W_\tria}}+\new{(I_\tria^\B|_{\W_\tria})'}  B_\tria^{\B} \new{I_\tria^\B|_{\W_\tria}},
$$
\new{where $I_\tria^\Ss$ is the linear projector defined on $\Ss_{\tria,0}^{0,1} \oplus \B_\tria$ by $\ran I_\tria^\Ss=\Ss^{0,1}_{\tria,0}$ and $\ran I_\tria^\B =\B_\tria$, where $I_\tria^\B:=\identity-I_\tria^\Ss$.}
 Then one has $\sup_{\tria\in \bbT}  \|B_\tria \|_{\cL(\W_\tria,\W_\tria')}<\infty$ and $\sup_{\tria\in \bbT}  \|\Re(B_\tria )^{-1}\|_{\cL(\W_\tria',\W_\tria)}<\infty$.

Substituting the definition of $\psi_{\tria,\nu}$, one infers that \framebox{$\bm{G}_\tria=\bm{D}_\tria^{-1} \bm{B}_\tria \bm{D}_\tria^{-\top}$}, where
$$
 \framebox{$\bm{B}_\tria=\bm{p}_\tria^\top \bm{B}_\tria^{\Ss} \bm{p}_\tria+\bm{q}_\tria^\top \bm{B}^{\B}_\tria \bm{q}_\tria$},
 $$
and 
\begin{align*}
&\framebox{$(\bm{q}_\tria)_{\nu' \nu}:=
\begin{cases}
 \frac{\delta_{\nu' \nu}}{d+1} & \nu \in N_\tria \setminus N^0_\tria,\\
\frac{d}{(d+2)(d+1)} & \nu \in N^0_\tria, \nu' =\nu,\\
-\frac{|\omega_\tria(\nu)\cap\omega_\tria(\nu')|}{(d+2)(d+1)|\omega_\tria(\nu')|} & \nu \in N^0_\tria, \nu' \neq \nu,
\end{cases}$}
\qquad
\framebox{$\bm{B}_\tria^{\Ss}:=\cF_{\Phi_\tria}' B_\tria^{\Ss}  \cF_{\Phi_\tria}$},
\\
&\framebox{$(\bm{p}_\tria)_{\nu' \nu}:=\delta_{\nu' \nu}\, (\nu' \in N^0_\tria,\,\nu \in N_\tria)$},\qquad
\framebox{$\bm{B}^{\B}_\tria :=\diag\{\beta  |\omega_\tria(\nu)|^{1-\frac{2s}{d}}\colon \nu \in N_\tria\}$}.
\end{align*}

\subsection{Manifold case.} \label{Smanilin}
From Sect.~\ref{Smani} recall the definitions of $\Gamma$, $\gamma$, $\W$, $\V$, $\new{\chi}\colon \Omega \rightarrow \cup_{i=1}^p \new{\chi}_i(\Omega_i)$, and that of the family of conforming partitions $\bbT$ of $\Gamma$.

As in the domain case discussed in Sect.~\ref{SlinearsWD}, for $\tria \in \bbT$  let $N_\tria$ be the set of vertices of $\tria$, and
$N^0_\tria$ its subset of vertices not on $\gamma$, for $T \in \tria$ let $N_T$ be the vertices of $T$,
$N^0_{\tria,T}:=N^0_\tria \cap N_T$, and for $\nu \in N_\tria$ let $\omega_\tria(\nu):=\cup_{\{T \in \tria\colon \nu \in N_T\}} \overline{T}$.

 We take
 \begin{align*}
\V_\tria=\Ss^{0,1}_\tria&:=\{u \in H^1(\Gamma)\colon u \circ \new{\chi}|_{\new{\chi}^{-1}(T)} \in \cP_1 \,(T \in \tria)\} \subset \V,\\
\Ss^{0,1}_{\tria,0}&:=\{u \in H^1_{0,\gamma}(\Gamma)\colon u \circ \new{\chi}|_{\new{\chi}^{-1}(T)} \in \cP_1 \,(T \in \tria)\},
 \end{align*}
 equipped with nodal bases $\Xi_\tria=\{\xi_{\tria,\nu}\colon \nu \in N_\tria\}$ and $\Phi_\tria=\{\phi_{\tria,\nu}\colon \nu \in N^0_\tria\}$, respectively,
defined by
 $$
  \xi_{\tria,\nu}(\nu')=\delta_{\nu,\nu'} \quad (\nu,\nu'\in N_\tria),
$$
and $\phi_{\tria,\nu}=\xi_{\tria,\nu}$ for $\nu \in N^0_\tria$.

Actually exclusively for the deriving an inverse inequality analogous to Lemma~\ref{inverse}, first we construct a collection $\Psi_\tria=\{\psi_{\tria,\nu}\colon \nu \in   N_\tria\}\subset H^1_{0,\gamma}(\Gamma)$ that has an `approximation property' and that is biorthogonal to $\Xi_\tria$ w.r.t.~the \emph{true} $L_2(\Gamma)$-scalar product:
We define $\Sigma_\tria=\{\sigma_{\tria,\nu}\colon \nu \in N_\tria\}$ by $\sigma_{\tria,\nu}:=\phi_{\tria,\nu}$ when $\nu \in N^0_\tria$, and $\sigma_{\tria,\nu}:=0$ when $\nu \in N_\tria \setminus  N^0_\tria$. Then, obviously, $\sum_{\nu \in N_\tria} \sigma_{\tria,\nu}$ equals $\1$ on $\Gamma \setminus \cup_{\{T \in \tria\colon \overline{T} \cap \gamma \neq \emptyset\}} \overline{T}$.

Given a $d$-simplex $T \subset  \R^d$, by means of an affine bijection we transport the function $\eta_\eps$, defined in Sect.~\ref{Smani} on a reference $d$-simplex $\hat{T}$, to a function on $T$ and denote it by $\eta_{T,\eps}$.
Then for any panel $T \in \tria \in \bbT$, for some $\eps>0$ small enough it holds that
$$
\inf_{0 \neq p \in \cP_1(\new{\chi}^{-1}(T))} \sup_{0 \neq q \in \cP_1(\new{\chi}^{-1}(T))} \frac{\langle p \circ \new{\chi}^{-1} ,(\eta_{T,\eps} q) \circ \new{\chi}^{-1} \rangle_{L_2(T)}}{\|p\circ \new{\chi}^{-1}\|_{L_2(T)}\|(\eta_{T,\eps} q) \circ \new{\chi}^{-1}\|_{L_2(T)}}>0
$$
Moreover, since the panels $T$ get increasingly flat when $\diam T \rightarrow 0$, there exists an $\eps>0$ such that above inf-sup condition is satisfied \emph{uniformly} over all $T \in \tria \in \bbT$.

By selecting for each $\nu \in  N_\tria$ a $T(\nu) \in \tria$ with $\nu \in N_T$, as in Sect.~\ref{SlinearsWD} we obtain a collection $\Theta_\tria =\{\theta_{\tria,\nu}\colon \nu \in N_\tria\}$ with $\theta_{\tria,\nu} \subset H^1_0(T(\nu))$
 that is biorthogonal to $\Xi_\tria$, in particular that satisfies \eqref{props}, after which we define the $\psi_{\tria,\nu}$ by means of formula \eqref{deflinpsi}.
Having constructed the biorthogonal collections $\Xi_\tria$ and $\Psi_\tria$, we set the
biorthogonal projector $P_\tria:L_2(\Gamma) \rightarrow H^1_{0,\gamma}(\Gamma)\colon
u \mapsto \sum_{\nu \in N_\tria} \frac{\langle u,\xi_{\tria,\nu} \rangle_{L_2(\Gamma)}}{\langle \1,\xi_{\tria,\nu}\rangle_{L_2(\Gamma)}} \psi_{\tria,\nu}$ which satisfies
$\| P_\tria u\|_{H^1(\Gamma)} \lesssim \|h_\tria^{-1} u\|_{L_2(\Gamma)}$. With the aid of this projector, as in Lemma~\ref{inverse} one infers that
\be \label{inverse2}
\|h_\tria v_\tria\|_{L_2(\Gamma)} \lesssim \|v_\tria\|_{(H^1_{0,\gamma})'} \quad (v_\tria \in \V_\tria).
\ee

Having established this inverse inequality, to arrive at a construction of $\Psi_\tria$ that does not require the evaluation of integrals over $\Gamma$, as in Sect.~\ref{Smodified} we replace $\langle\,,\,\rangle_{L_2(\Gamma)}$ by $\langle\,,\,\rangle_{\tria}$.
We redefine $\Theta_\tria=\{\theta_{\tria,\nu}\colon \nu \in N_\tria\}$ by
$$
\theta_{\tria,\nu}:=|\omega_\tria(\nu)| {\textstyle \frac{|\new{\chi}^{-1}(T)|}{|T|}} \tilde{\phi}_{\new{\chi}^{-1}(T),\eps,\new{\chi}^{-1}(\nu)}\circ \new{\chi}^{-1}
$$
with the $\tilde{\phi}$'s defined in \eqref{phitilde}, and following \eqref{deflinpsi} set $\Psi_\tria=\{\psi_{\tria,\nu}\colon \nu \in N_\tria\}$ and $\W_\tria:=\Span \Psi_\tria$ by
$$
\psi_{\tria,\nu}= \left\{
\begin{array}{@{}ll@{}}
{\textstyle \frac{\theta_{\tria,\nu}}{d+1}} & \nu \in N_\tria \setminus N^0_\tria,\\
\phi_{\tria,\nu}+\frac{d}{(d+2)(d+1)}\theta_{\tria,\nu}-\sum_{\nu' \in N_\tria\setminus\{\nu\}} \frac{|\omega_\tria(\nu)\cap\omega_\tria(\nu')|}{(d+2)(d+1)|\omega_\tria(\nu')|}\theta_{\tria,\nu'}& \nu \in  N^0_\tria,
\end{array}
\right.
$$
As in Sect.~\ref{Smodified}, we set
$(D_\tria v_\tria)(w_\tria):=\langle v_\tria,w_\tria\rangle_\tria$ ($v_\tria \in \V_\tria,\,w_\tria \in \W_\tria$), and
as in Sect.~\ref{Smodified}, using \eqref{inverse2} one shows that $\sup_{\tria \in \bbT} \|D_\tria\|_{\cL(\V_\tria,\W_\tria')} <\infty$.
Similarly as in Lemma~\ref{lem4}, one proves that
$$
\|h_\tria v_\tria\|_{L_2(\Gamma)} \lesssim \sup_{0 \neq w \in H^1_{0,\gamma}(\Gamma)} \frac{\langle v_\tria,w_\tria\rangle_\tria}{\|w\|_{H^1(\Gamma)}},
$$
and with that $\sup_{\tria \in \bbT} \|D^{-1}_\tria\|_{\cL(\W_\tria',\V_\tria)} <\infty$.

Constructing $B_\tria \in \cLis(\W_\tria,\W_\tria')$ as in Sect.~\ref{Bmani}, one arrives at the same expressions for $\bm{D}_\tria$, $\bm{G}_\tria$, $\bm{B}_\tria$, $\bm{q}_\tria$,
$\bm{B}_\tria^{\Ss}$,
$\bm{p}_\tria$, and
$\bm{B}^{\B}_\tria$
as in Sect.~\ref{SlinearsWD}-\ref{Bmani} in the domain case.


\section{Higher order case} \label{Shigherorder}
In this section, we discuss the construction of an uniform preconditioner
for  $\V_\tria$ being either the space $\Ss_\tria^{-1,\ell}$ of \emph{discontinuous} piecewise polynomials of degree $\ell>0$ w.r.t. $\tria$, 
or the space $\Ss_\tria^{0,\ell}$ of \emph{continuous} piecewise polynomials of degree $\ell>1$ w.r.t. $\tria$.

We write the spaces $\V_\tria$, $\B_\tria$, $\W_\tria$, and their bases
$\Xi_\tria$, $\Theta_\tria$,
$\Psi_\tria$ from Sect.~\ref{Sdomain_case} or \ref{Slinears} as
 $\V^0_\tria$, $\B^0_\tria$, $\W^0_\tria$, and
$\Xi^0_\tria$, $\Theta^0_\tria$,
$\Psi^0_\tria$, respectively.
The biorthogonal projector formerly denoted as $P_\tria$ will now be denoted as  $P^0_\tria$, and the matrices 
${\bf B}_\tria$ and ${\bf D}_\tria$ as  ${\bf B}^0_\tria$ and ${\bf D}^0_\tria$.

Although we consider the domain case, the results extend to the manifold case following the approach outlined in \S\ref{Smani} or \S\ref{Smanilin}.

In order to construct an uniform preconditioner, obvious possibilities are to apply the framework of operator preconditioning directly to the higher order polynomial space $\V_\tria$, or to use the preconditioner developed for the lowest order case within a subspace correction framework. We investigate both possibilities.

\subsection{Application of the operator preconditioning framework}

\subsubsection{Discontinuous piecewise polynomials} \label{Sdiscontinuous}
Given $\ell>0$, for $\tria \in \bbT$, let $\V_\tria=\Ss_\tria^{-1,\ell}$.
With $m=m(\ell):={d+\ell\choose \ell}-1$, we equip $\V_\tria$ with $\Xi_\tria=\{\xi_{T,i}\colon T \in \tria,\,0 \leq i \leq m\}$, where for each $T \in \tria$, 
$\{\xi_{T,i}\colon 0 \leq i \leq m\}$ is constructed by the common affine lifting approach from a basis for the polynomials of degree $\ell$ on a reference $d$-simplex, such that
$\{\xi_{T,0}\colon T \in \tria\}=\Xi_\tria^0$, $\supp \xi_{T,i} \subset \overline{T}$, and $\|\xi_{T,i}\|_{L_2(\Omega)} \eqsim |T|^{\frac12}$.\footnote{For $i>1$, it is allowed that $\xi_{T,i}$ is (nearly) orthogonal to $\1|_T$, i.e., \eqref{e4} is not required for these  $\xi_{T,i}$.}

A straightforward generalization of the construction in the first paragraphs of \S\ref{SlinearsWD}
of a collection in $H^1_0(T)$ that is biorthogonal to the nodal basis of $P_1(T)$ shows the following:
There exists a set of `bubbles'  $\Theta_\tria = \{\theta_{T,i}\colon T \in \tria,\,0 \leq i \leq m\} \subset H^1_{0,\gamma}(\Omega)$ such that for $T,T' \in \tria$, $0 \leq i, i' \leq m$, $k \in \{0,1\}$,
\begin{align} \label{g0}
&\delta_{(T,i),(T',i')} |T|=\langle \theta_{T,i},\xi_{T',i'}\rangle_{L_2(\Omega)} \eqsim \delta_{(T,i),(T',i')} \|\theta_{T,i}\|_{L_2(\Omega)} \|\xi_{T',i'}\|_{L_2(\Omega)},\\
&\|\theta_{T,i}\|_{H^1(\Omega)} \lesssim h_T^{-1} \|\theta_{T,i}\|_{L_2(\Omega)},\\ 
&\sup \theta_{T,i} \subset \overline{T},\\ \label{g1}
&\big\|\sum_{\{T \in \tria, 0 \leq i \leq m\}} c_{T,i} \theta_{T,i}\big\|^2_{H^k(\Omega)} \eqsim \sum_{\{T \in \tria, 0 \leq i \leq m\}} h_T^{-2k} |c_{T,i}|^2 \| \theta_{T,i}\|_{L_2(\Omega)}^2,\\ \label{g2}
&\|u+v\|_{H^k(\Omega)}^2 \gtrsim \|u\|_{H^k(\Omega)}^2+\|v\|_{H^k(\Omega)}^2 \quad(u \in \Ss^{0,1}_{\tria,0},\,v \in \B_\tria:=\Span \Theta_\tria),\\ \label{g3}
&\{\theta_{T,0}\colon T \in \tria\}=\Theta_\tria^0.
\end{align}

Writing $\Psi_\tria^0=\{\psi^0_{\tria,T}\colon T \in \tria\}$, we define $\Psi_\tria:=\{\psi_{\tria,T,i}\colon T \in \tria,\,0\leq i \leq m\}$, $\W_\tria:=\Span \Psi_\tria$ by
$$
\psi_{\tria,T,i}:=\left\{
\begin{array}{ll} \psi_{\tria,T}^0-\sum_{\{T' \in \tria,1\leq i'\leq m\}} \frac{\langle \psi^0_{\tria,T}, \xi_{T',i'}\rangle_{L_2(\Omega)}}{\langle \theta_{T',i'}, \xi_{T',i'}\rangle_{L_2(\Omega)}} \theta_{T',i'} & i=0,\\
\theta_{T,i} & 1 \leq i \leq m.
\end{array}
\right.
$$
Knowing that $\Psi_\tria^0$ and $\Xi_\tria^0$, and $\Theta_\tria$ and $\Xi_\tria$ are biorthogonal, the correction made to the $\psi_{\tria,T}^0$ ensures that  $\Psi_\tria$ and $\Xi_\tria$ are biorthogonal, in particular that
$$
\langle \psi_{\tria,T,i},\xi_{T',i'}\rangle=\delta_{(T,i),(T',i')} |T|.
$$
For use later, notice that $\W_\tria^0 \subset \W_\tria$, and that by definition of $\psi^0_{\tria,T}$ and $\sigma_{\tria,T}$, for $i'>0$,
\be \label{defr}
\begin{split}
({\bf R}_\tria)_{(T',i'),T}&:=-\frac{\langle \psi^0_{\tria,T}, \xi_{T',i'}\rangle_{L_2(\Omega)}}{\langle \theta_{T',i'}, \xi_{T',i'}\rangle_{L_2(\Omega)}}=-|T'|^{-1} \langle \sigma_{\tria,T}, \xi_{T',i'}\rangle_{L_2(\Omega)}\\
&=-|T'|^{-1} \sum_{\nu \in N^0_{\tria,T} \cap N^0_{\tria,T'}} d_{\tria,\nu}^{-1} \langle \phi_{\tria,\nu}, \xi_{T',i'}\rangle_{L_2(\Omega)}.
\end{split}
\ee

The biorthogonal projector $P_\tria\colon L_2(\Omega) \rightarrow H^1_{0,\gamma}(\Omega)$ 
with $\ran P_\tria=\W_\tria$ and \linebreak$\ran(\identity-P_\tria)=\V_\tria^{\perp_{L_2(\Omega)}}$
 is given by
$$
P_\tria u=\sum_{\{T\in \tria,0 \leq i \leq m\}} \frac{\langle u,\xi_{T,i}\rangle_{L_2(\Omega)}}{\langle \psi_{\tria,T,i},\xi_{T,i}\rangle_{L_2(\Omega)}}\psi_{\tria,T,i}.
$$

\begin{theorem} \label{thm10} It holds that $\sup_{\tria \in \bbT} \|P_\tria\|_{\cL(\W,\W)}<\infty$.
\end{theorem}

\begin{proof}
For $T'' \in \tria$, $k \in \{0,1\}$, from $\supp \xi_{T,i} \subset \overline{T}$ and $\supp \psi_{\tria,T'',i} \subset \omega_\tria^{(1)}(T'')$ we have
\be \label{f1}
\|P_\tria u\|_{H^k(T'')} \leq \sum_{\{T \in R_\tria^{(1)}(T''),\,0 \leq i \leq m\}} \|u\|_{L_2(T)} \frac{\|\psi_{\tria,T,i}\|_{H^k(\Omega)}\|\xi_{T,i}\|_{L_2(\Omega)}}{|\langle \psi_{\tria,T,i},\xi_{T,i}\rangle_{L_2(\Omega)}|}.
\ee
To bound the right-hand side we distinguish between terms with $i=0$ and those with $i>0$.

From $\|\psi_{\tria,T}^0\|_{H^k(\Omega)} \lesssim h_T^{d/2-k}$ (\eqref{e6}), and for $T' \in \tria$, $1 \leq i' \leq m$,
$$ \textstyle
\big\|\frac{\langle \psi^0_{\tria,T}, \xi_{T',i'}\rangle_{L_2(\Omega)}}{\langle \theta_{T',i'}, \xi_{T',i'}\rangle_{L_2(\Omega)}} \theta_{T',i'}\big\|_{H^k(\Omega)}
\lesssim \frac{\|\psi^0_{\tria,T}\|_{L_2(\Omega)}\| \xi_{T',i'}\|_{L_2(\Omega)}}{\|\theta_{T',i'}\|_{L_2(\Omega)}\| \xi_{T',i'}\|_{L_2(\Omega)}} h_{T'}^{-k}\|\theta_{T',i'}\big\|_{L_2(\Omega)}\lesssim h_T^{d/2-k},
$$
we infer that
\be \label{f2}
\|\psi_{\tria,T,0}\|_{H^k(\Omega)} \lesssim h_T^{d/2-k},
\ee
while, thanks to \eqref{e4},
\be \label{f3}
\frac{\|\xi_{T,0}\|_{L_2(\Omega)}}{|\langle \psi_{\tria,T,0},\xi_{T,0}\rangle_{L_2(\Omega)}|}=
\frac{\|\xi_{T,0}\|_{L_2(\Omega)}}{|\langle \psi_{\tria,T}^0,\xi_{T,0}\rangle_{L_2(\Omega)}|}=
\frac{\|\xi_{T,0}\|_{L_2(\Omega)}}{|\langle \1,\xi_{T,0}\rangle_{L_2(\Omega)}|} \eqsim h_T^{-d/2}.
\ee

For $1 \leq i \leq m$,
\be \label{f4}
\frac{\|\psi_{\tria,T,i}\|_{H^k(\Omega)}\|\xi_{T,i}\|_{L_2(\Omega)}}{|\langle \psi_{\tria,T,i},\xi_{T,i}\rangle_{L_2(\Omega)}|}
=
\frac{\|\theta_{T,i}\|_{H^k(\Omega)}\|\xi_{T,i}\|_{L_2(\Omega)}}{|\langle \theta_{T,i},\xi_{T,i}\rangle_{L_2(\Omega)}|}
\eqsim 
\frac{\|\theta_{T,i}\|_{H^k(\Omega)}}{\|\theta_{T,i}\|_{L_2(\Omega)}} \lesssim h_T^{-k}.
\ee

From \eqref{f1}-\eqref{f4}, we conclude that
\be \label{f5}
\|P_\tria u\|_{H^k(T'')} \lesssim h_T^{-k} \|u\|_{L_2(\omega_\tria^{(1)}(T''))},
\ee
which is the analogue of estimate \eqref{10} that was proven for $P_\tria^0$.
Since $\W_\tria^0 \subset \W_\tria$, by making use of the \emph{same} Scott-Zhang type quasi-interpolator $\Pi_\tria$ as has been used in the proof of Theorem~\ref{theorem1}, copying the remainder of that proof  yields the claim.
\end{proof}

Thanks to above theorem, $D_\tria\colon \V_\tria \rightarrow \W_\tria'$ defined by $(D_\tria v)(w)=\langle v,w\rangle_{L_2(\Omega)}$
satisfies $\sup_{\tria \in \bbT}\max\big(\|D_\tria\|_{\cL(\V_\tria,\W_\tria')},\|D^{-1}_\tria\|_{\cL(\W_\tria',\V_\tria)}\big)<\infty$.
By enumerating all $\xi_{T,i}$ and $\psi_{\tria,T,i}$ with index $i=0$ before those with index $i>0$, its
matrix representation reads as 
${\bf D}_\tria=\left[\begin{array}{@{}cc@{}} {\bf D}_\tria^0 & 0 \\ 0 & {\bf D}_\tria^1\end{array} \right]$, where 
${\bf D}_\tria^1:=\diag\{|T| \identity_{m \times m} \colon T \in \tria\}$.

Thanks to \eqref{g1}-\eqref{g2}, a suitable $B_\tria \in \cLis(\W_\tria,\W_\tria')$ can be defined similarly as in \S\ref{SB}:
With $I_\tria^\Ss$ being the linear projector defined on $\Ss_{\tria,0}^{0,1} \oplus \B_\tria$ by $\ran I_\tria^\Ss=\Ss^{0,1}_{\tria,0}$ and $\ran  I_\tria^\B  =\B_\tria$, where $I_\tria^\B:=\identity-I_\tria^\Ss$,
we define
$B^{\Ss\oplus\B}_\tria\in \cLis(\Ss^{0,1}_{\tria,0}\oplus \B_\tria,(\Ss^{0,1}_{\tria,0} \oplus \B_\tria)')$ by
 $$
 (B^{\Ss\oplus\B}_\tria w)(\tilde w)=(B^{\Ss}_\tria I_\tria^\Ss w)(I_\tria^\Ss \tilde w)+(B^\B_\tria I_\tria^\B w)(I_\tria^\B \tilde w),
 $$
 where
 $$
 (B_\tria^\B \sum_{\{T \in \tria,\,0\leq i \leq m\}} c_{T,i} \theta_{\tria,i})(\sum_{\{T \in \tria,\,0\leq i \leq m\}} d_{T,i} \theta_{\tria,i}):=\beta \sum_{\{T \in \tria,\,0\leq i \leq m\}} h_T^{d-2s} c_{T,i} d_{T,i},
 $$
and $B_\tria^\Ss$ is as in \S\ref{SB}, 
 and define
$B_\tria \in \cLis(\W_\tria,\W_\tria')$ by $(B_\tria w)(\tilde w)=(B^{\Ss\oplus\B}_\tria w)(\tilde w)$.

Using that with ${\bf R}_\tria$ as defined in \eqref{defr}, $\left[\begin{array}{@{}cc@{}} \identity & 0 \\ {\bf R}_\tria & \identity \end{array} \right]$ is the basis transformation from  
$\Psi_\tria$ to $\Psi_\tria^0 \cup \{\theta_{T,i}\colon T \in \tria,\,1 \leq i \leq m\}$,
one infers that the representation of the resulting uniform preconditioner reads as
\be \label{precondhigh}
{\bf G}_\tria=
\left[\begin{array}{@{}cc@{}} {\bf D}_\tria^0 & 0 \\ 0 & {\bf D}_\tria^1\end{array} \right]^{-1}
\left[\begin{array}{@{}cc@{}} \identity & {\bf R}_\tria^\top \\ 0 & \identity \end{array} \right]
\left[\begin{array}{@{}cc@{}} {\bf B}_\tria^0 & 0 \\ 0 & \beta({\bf D}_\tria^1)^{1-\frac{2s}{d}}\end{array} \right]
\left[\begin{array}{@{}cc@{}} \identity & 0 \\  {\bf R}_\tria & \identity \end{array} \right]
\left[\begin{array}{@{}cc@{}} {\bf D}_\tria^0 & 0 \\ 0 & {\bf D}_\tria^1\end{array} \right]^{-1},
\ee
which is thus independent of the particular bubbles $\Theta_\tria$ being chosen.

\subsubsection{Continuous piecewise polynomials} \label{Scontinuous}
Given $\ell>1$, for $\tria \in \bbT$, let $\V_\tria=\Ss_\tria^{0,\ell}$. 
We equip $\V_\tria$ with a basis $\Xi_\tria=\Xi_\tria^0 \cup (\Xi_\tria \setminus \Xi_\tria^0)$, where for each $T \in \tria$, the set of restrictions to $T$ of those basis functions that do not identically vanish on $T$ is a lifted version of a fixed basis for the polynomials of degree $\ell$ on a reference $d$-simplex under an affine bijection; the support of each basis function $\xi \in \Xi_\tria$ is connected and extends to a uniformly bounded number of  $T \in \tria$; and finally, $\|\xi\|_{L_2(\Omega)}\eqsim |\supp \xi|^{\frac12} \eqsim |T|^{\frac12}$ for some $T \in \tria$ with $\supp \xi \cap T \neq \emptyset$.

Similar to the previous \S\ref{Sdiscontinuous}, a biorthogonal collection $\Theta_\tria=\{\theta(\xi)\colon \xi \in \Xi_\tria\}$ of `bubbles' exists that has properties 
analogous to \eqref{g0}-\eqref{g3}, reading $|T|$ as $|\supp \xi|$, and $h_T$ as $|\supp \xi|^{1/d}$.
Writing the collection $\Psi^0_\tria$ as found in \S\ref{Slinears} as $\{\psi^0(\xi)\colon \xi \in \Xi_\tria^0\}$, we define $\Psi_\tria=\{\psi(\xi)\colon \xi \in \Xi_\tria\}$, biorthogonal to $\Xi_\tria$, and $\W_\tria:=\Span \Psi_\tria$ by
$$
\psi(\xi):=\left\{
\begin{array}{ll} \psi^0(\xi)-\sum_{\xi' \in \Xi_\tria \setminus \Xi_\tria^0} \frac{\langle \psi^0(\xi), \xi'\rangle_{L_2(\Omega)}}{\langle \theta(\xi'),\xi' \rangle_{L_2(\Omega)}} \theta(\xi') & \xi \in \Xi_\tria^0,\\
\theta(\xi)& \xi \in \Xi_\tria \setminus \Xi_\tria^0.
\end{array}
\right.
$$
Theorem~\ref{thm10} extends to the current setting and shows that the corresponding biorthogonal projector is uniformly bounded. The representation of resulting uniform preconditioner reads as \eqref{precondhigh}, obviously now with ${\bf D}_\tria^0$ and ${\bf B}_\tria^0$ as found in the continuous piecewise linear case, and the matrix 
${\bf D}_\tria^1:=\diag\{|\supp \xi|\colon \xi \in \Xi_\tria \setminus \Xi_\tria^0\}$, and $({\bf R}_\tria^1)_{\xi,\nu}:=-|\supp \xi|^{-1}\langle\phi_{\tria,\nu},\xi\rangle_{L_2(\Omega)}$ for $(\xi,\nu) \in (\Xi_\tria \setminus \Xi_\tria^0) \times N^0_\tria$, and $({\bf R}_\tria^1)_{\xi,\nu}=0$ for $\nu \in N_\tria\setminus N^0_\tria$ (recall $\# \Xi_\tria^0=\# N_\tria$).

\subsection{Application of a subspace correction framework}
For $\V_\tria=\Ss_\tria^{-1,\ell}$, in \S\ref{Sdiscontinuous} we have demonstrated existence of a biorthogonal projector that satisfies \eqref{f5}.
In  \S\ref{Scontinuous} we have shown that a similar result holds true for $\V_\tria=\Ss_\tria^{0,\ell}$.
From the proof of Lemma~\ref{inverse} we learn that this implies that for either choice of  $\V_\tria$,
\be \label{inverseagain}
\|h_\tria v_\tria\|_{L_2(\Omega)} \lesssim \|v_\tria\|_{(H_{0,\gamma}^1)'}\quad (v_\tria \in \V_\tria).
\ee
Using this inverse inequality, we are going to decompose $\V_\tria$ in a uniformly stable way into $\V_\tria^0$
and a complement space on which the $\|\,\|_\V$-norm is equivalent to a scaled $L_2(\Omega)$-norm.

\begin{proposition} \label{prop11} Let $Q_\tria^0 \in \cL(L_2(\Omega),L_2(\Omega))$ be a projector with $\ran Q_\tria^0=\V^0_\tria$, \\
$\sup_{\tria \in \bbT} \|Q_\tria^0\|_{\cL(L_2(\Omega),L_2(\Omega))}<\infty$,
and
$$
\sup_{\tria \in \bbT} \|h_\tria^{-1}(\identity-(Q_\tria^0)^*)\|_{\cL(H^1_{0,\gamma}(\Omega),L_2(\Omega))}<\infty.
$$
Then $\sup_{\tria \in \bbT} \|Q_\tria^0|_{\V_\tria}\|_{\cL(\V_\tria,\V_\tria)}<\infty$, and
\be \label{equiv}
\|h_\tria^s \cdot\|_{L_2(\Omega)} \eqsim \|\cdot\|_{\V}\quad \text{on } \V_\tria^1:=\ran(\identity-Q_\tria^0)|_{\V_\tria}.
\ee
\end{proposition}

\begin{proof} For $u \in \V_\tria$, thanks to \eqref{inverseagain} we have
\begin{align} \label{firstline}
\|(\identity - Q^0_\tria)u\|_{H^1_{0,\gamma}(\Omega)'}&=\sup_{v \in H^1_{0,\gamma}(\Omega)} \frac{\langle u,(\identity - (Q^0_\tria)^*)v\rangle_{L_2(\Omega)}}{\|v\|_{H^1(\Omega)}} \lesssim 
\|h_\tria u\|_{L_2(T)}\\ \nonumber
& \lesssim \|u\|_{H^1_{0,\gamma}(\Omega)'},
\end{align}
so that the first statement follows by interpolation.

Again by interpolation, the second statement needs only to be proven for $s=1$. For that case it follows from \eqref{firstline} and \eqref{inverseagain}.
 \end{proof}
 
 Next we use the decomposition $\V_\tria=\V^0_\tria \oplus \V^1_\tria$ from Proposition~\ref{prop11} to build a preconditioner on $\V_\tria$ from preconditioners on the subspaces.

\begin{proposition} In the situation of Proposition~\ref{prop11}, for $i=0,1$, let $I^i_\tria$ denote the embedding of $\V_\tria^i$ into $\V_\tria$, and let $G_\tria^i \in \cLis((\V_\tria^i)',\V_\tria^i)$.
Then $G_\tria:=\sum_{i=0}^1 I^i_\tria G_\tria^i (I_\tria^i)' \in \cLis((\V_\tria)',\V_\tria)$ with
\begin{align*}
\|G_\tria\|_{\cL(\V_\tria',\V_\tria)} &\leq 2 \max_i \|G^i_\tria\|_{\cL((\V^i_\tria)',\V^i_\tria)},\\
\|\Re(G_\tria)^{-1}\|_{\cL(\V_\tria,\V'_\tria)} &\leq 2 \|Q_\tria^0|_{\V_\tria}\|^2_{\cL(\V_\tria,\V_\tria)}\max_i \|\Re(G^i_\tria)^{-1}\|_{\cL(\V^i_\tria,(\V^i_\tria)')}
\end{align*}
\end{proposition}

\begin{proof} The result follows as an easy case from the general theory of (additive) subspace correction methods (e.g.~\cite{242.5} + references cited there), together with the inequalities
$$
\frac12 \|\cdot\|_{\V}^2 \leq \|Q_\tria^0 \cdot\|_{\V}^2+\|(\identity-Q_\tria^0)\cdot\|_{\V}^2 \leq 2  \|Q_\tria^0\|_{\cL(\V_\tria,\V_\tria)} \|\cdot\|_{\V}^2 \quad \text{on }\V_\tria,
$$
where we used that $\|(\identity-Q_\tria^0)|_{\V_\tria}\|_{\cL(\V_\tria,\V_\tria)}=\|Q_\tria^0|_{\V_\tria}\|_{\cL(\V_\tria,\V_\tria)}$, because $Q_\tria^0|_{\V_\tria}$ is a projector unequal to both the zero map and the identity (\cite{169.5,315.7})
\end{proof}

On $\V^0_\tria$ we already have our uniform preconditioner $G_\tria^0$ available, so it remains to construct such a preconditioner on the complement space $\V_\tria^1$.
In the situation of Proposition~\ref{prop11},  
 let $\Xi_\tria=\Xi^0_\tria \cup (\Xi_\tria\setminus \Xi_\tria^0)$ be a basis for $\V_\tria$ such that $\Xi_\tria^1:=\Xi_\tria\setminus \Xi_\tria^0$ is a basis for $\V_\tria^1$  for which 
 \be \label{basiscondition}
\|h_\tria^s\sum_{\xi \in \Xi^1_\tria}  c_\xi \xi\|_{L_2(\Omega)}^2 \eqsim \sum_{\xi \in \Xi^1_\tria} |c_\xi|^2 \|h_\tria^s \xi\|_{L_2(\Omega)}^2.
\ee
Then 
$$
(H_\tria^1 \sum_{\xi \in \Xi_\tria^1} c_\xi \xi)(\sum_{\xi \in \Xi_\tria^1} d_\xi \xi):=\sum_{\xi \in \Xi_\tria^1} c_\xi d_\xi \|h_\tria^s \xi\|_{L_2(\Omega)}^2
$$
satisfies $\sup_{\tria \in \bbT}\max(\|H^1_\tria\|_{\cL(\V^1_\tria,(\V^1_\tria)')}, \|\Re(H^1_\tria)^{-1}\|_{\cL((\V^1_\tria)',\V_\tria)})<\infty$ thanks to \eqref{equiv}, and so $G^1_\tria:=(H_\tria^1)^{-1}$ is a suitable choice.
The implementation of the resulting uniform preconditioner $G_\tria$ reads as
\be \label{33}
{\bf G}_\tria=\left[\begin{array}{@{}cc@{}} {\bf G}_\tria^0 & 0 \\ 0 & \diag\{\|h_\tria^s \xi\|_{L_2(\Omega)}^{-2}\colon \xi \in \Xi_\tria^1\}\end{array} \right].
\ee

What remains is,  for both options for $\V_\tria$, to specify a $Q_\tria^0$ that satisfies the conditions from Proposition~\ref{prop11}, and to equip $\V_\tria$ with a basis $\Xi_\tria$ that is the union of the basis $\Xi_\tria^0$ for $\V_\tria^0$, and a basis  $\Xi_\tria^1$ for  $\V_\tria^1=\ran(\identity-Q_\tria^0)|_{\V_\tria}$, the latter being \emph{uniformly stable} w.r.t. $\|h_\tria^s \cdot\|_{L_2(\Omega)}$, i.e., one that satisfies \eqref{basiscondition}.

\subsubsection{Discontinuous piecewise polynomials} \label{discontinuoussubspacecorrection}
For $\V_\tria=\Ss_\tria^{-1,\ell}$, the conditions of Proposition~\ref{prop11} are fulfilled by taking $Q_\tria^0$ to be the $L_2(\Omega)$-orthogonal projector onto $\V_\tria^0=\Ss_\tria^{-1,0}$.
\footnote{Notice that for $s \geq \frac12$, $Q_\tria^0 \not\in \cL(\V,\V)$, so taking its restriction to $\V_\tria$  is essential  in Proposition~\ref{prop11}.}

By taking $\Xi_\tria=\{\xi_{T,i}\}_{T \in \tria,0\leq i \leq m}$ to be an $L_2(\Omega)$-orthogonal basis for $\V_\tria$ such that
$\xi_{\tria,0}=\xi_T$ and $\supp \xi_{\tria,i} \subset \overline{T}$, \eqref{basiscondition} is valid.

Remarkably, with these specifications and by scaling $\|\xi_{\tria,i}\|_{L_2(\Omega)}=|T|^{\frac12}$, 
 the resulting ${\bf G}_\tria$ is given by \eqref{precondhigh} with ${\bf R}_\tria$ reading as the zero map (${\bf R}_\tria$ from \eqref{defr} is non-zero).

\subsubsection{Continuous piecewise polynomials} 
Let $\V_\tria=\Ss_\tria^{0,\ell}$ for $\ell>1$.
It is no option to take $Q_\tria^0$ to be the $L_2(\Omega)$-orthogonal projector onto $\V_\tria^0=\Ss_\tria^{0,1}$
because in that case we will not be able to equip $\ran(\identity-Q_\tria^0)|_{\V_\tria}$ with a locally supported basis.

From \eqref{biorthprojlin} recall the biorthogonal projector $P_\tria^0$ onto $\W_\tria^0$ with $\ran(\identity-P_\tria^0)=\V_\tria^0$. Writing $(\identity-P_\tria^0)=(\identity-P_\tria^0)(\identity-\Pi_\tria^0)$ with $\Pi_\tria^0$ being a Scott-Zhang type interpolator, one shows that 
$\sup_{\tria \in \bbT} \|h_\tria^{-1}(\identity-P_\tria^0)\|_{\cL(H^1_{0,\gamma}(\Omega),L_2(\Omega))}<\infty$.
Since furthermore $\sup_{\tria \in \bbT} \|P_\tria^0\|_{\cL(L_2(\Omega),L_2(\Omega))}<\infty$, the conditions of Proposition~\ref{prop11} are satisfied by
taking $Q_\tria^0:=(P_\tria^0)^*=u \mapsto\sum_{\nu \in N_\tria} \frac{ (d+1)\langle u, \psi^0_{\tria,\nu}\rangle_{L_2(\Omega)}}{|\omega_\tria(\nu)|}\xi^0_{\tria,\nu}$.

Let us denote the \emph{weighted} $L_2(\Omega)$-norm $\|h_\tria^s \cdot\|_{L_2(\Omega)}$ by $\nrm\cdot\nrm$.
We need to equip $\V_\tria^1=\ran(\identity-Q_\tria^0)|_{\V_\tria}$ with a basis that is uniformly stable w.r.t. $\nrm\cdot\nrm$. Since the supports of the basis functions will extend to multiple $T \in \tria$, this task is more complex than for the discontinuous piecewise polynomial case. Let $\Xi_\tria=\Xi^0_\tria \cup (\Xi_\tria\setminus \Xi^0_\tria)$ be a basis for $\V_\tria$ such that for each $T \in \tria$,
$\{\xi|_T\colon \xi \in \Xi,\, \xi|_T \not\equiv 0\}$ is a uniformly $L_2(T)$-stable basis for its span.
Common affine equivalent constructions yield such a basis.
Then $\Xi_\tria$ is stable w.r.t. $\nrm\cdot\nrm$, uniformly in $\tria \in \bbT$, and so in particular, with $\bar{\V}_\tria^1:=\Span \Xi_\tria\setminus \Xi^0_\tria$, $\V_\tria=\V_\tria^0 \oplus \bar{\V}_\tria^1$ is a uniformly stable decomposition w.r.t. $\nrm\cdot\nrm$.

Corresponding to this decomposition, for $v \in \V_\tria$ we write $v=v^0+\bar{v}^1$.
Taking $v \in \V_\tria^1$, it holds that 
$0=Q_\tria^0 v=Q_\tria^0 v^0+Q_\tria^0 \bar{v}^1=v^0+Q_\tria^0 \bar{v}^1$,
or $v^0=-Q_\tria^0 \bar{v}^1$ i.e. $v=(\identity-Q_\tria^0)\bar{v}^1$, showing that $\identity-Q_\tria^0 \colon \bar{\V}_\tria^1 \rightarrow \V_\tria^1$ is surjective. Injectivity of this map follows from $\nrm v\nrm \eqsim \nrm Q_\tria^0 \bar{v}^1\nrm+\nrm \bar{v}^1\nrm \geq \nrm \bar{v}^1\nrm$, and bounded invertibility, uniformly in $\tria$, will follow from 
$Q_\tria^0$ being uniformly bounded w.r.t. $\nrm\cdot\nrm$. 
The latter holds true because of  $\|\psi^0_{\tria,\nu}\|_{L_2(\Omega)}\|\xi^0_{\tria,\nu}\|_{L_2(\Omega)}\eqsim \frac{|\omega_\tria(\nu)|}{d+1}$, the local supports of the $\psi^0_{\tria,\nu}$ and $\xi^0_{\tria,\nu}$, and the uniform K-mesh property of $\tria$.
We conclude that $(\identity-Q_\tria^0)(\Xi_\tria\setminus \Xi^0_\tria)$ is a basis for $\V_\tria^1$ that is uniformly stable w.r.t. $\nrm\cdot\nrm$.

Since $\Xi_\tria^0 \cup (\identity-Q_\tria^0)(\Xi_\tria\setminus \Xi^0_\tria)$ is not the basis of choice to set up the stiffness matrix, we give the implementation ${\bf G}_\tria$ of the uniform preconditioner $G_\tria$ for $\V_\tria$ being equipped with $\Xi_\tria$, partitioned into $\Xi^0_\tria=\{\xi_{\tria,\nu}\colon \nu \in N_\tria\}$ and $\Xi_\tria\setminus \Xi_\tria^0$.
It reads as
$$
{\bf G}_\tria=
\left[\begin{array}{cc}
\identity & {\bf S}_\tria\\
0 & \identity
\end{array}
\right]
\left[\begin{array}{cc}
{\bf G}_\tria^0 & 0 \\
0 & \diag\{\|h_\tria^s \xi\|_{L_2(\Omega)}^{-2}\colon \xi \in \Xi_\tria \setminus \Xi_\tria^0\}
\end{array}
\right]
\left[\begin{array}{cc}
\identity & 0\\
{\bf S}_\tria^\top & \identity
\end{array}
\right],
$$
where for $(\nu,\xi) \in N_\tria \times (\Xi_\tria \setminus \Xi_\tria^0)$,  $({\bf S}_\tria)_{\nu \xi}:=-\frac{ (d+1)\langle \xi, \psi^0_{\tria,\nu}\rangle_{L_2(\Omega)}}{|\omega_\tria(\nu)|}$,
and where in \eqref{33} we have replaced 
$\|h_\tria^s (\identity-Q_\tria^0)\xi\|_{L_2(\Omega)}^{-2}$
by the equivalent
$\|h_\tria^s \xi\|_{L_2(\Omega)}^{-2}$.


\section{Numerical Experiments} \label{Snumerics}
Let $\Gamma = \partial [0,1]^3 \subset \R^3$ be the two-dimensional manifold without boundary given as the boundary of the unit cube, $\W := H^{1/2}(\Gamma)$, $\V := H^{-1/2}(\Gamma)$, and $\V_\tria =\Ss_\tria^{-1,\ell}\subset \V$.

The role of the opposite order operator $B_\tria^{\Ss}  \in \cLis(\Ss^{0,1}_{\tria,0},(\Ss^{0,1}_{\tria,0})')$ from Section~\ref{SB} will be fulfilled by
$(B_\tria^{\Ss} u)(v):=(Bu)(v)$ for
an adapted hypersingular operator $B \in \cLis(\W, \W')$. The hypersingular operator $\tilde B \in \cL(\W,\W')$ itself
 is only semi-coercive, but there
are various options to change it into a coercive operator (\cite{249.15}). We consider $B \in \cLis(\W,\W')$ given by
$(Bu)(v) = (\tilde Bu)(v) + \alpha \langle u, \1 \rangle_{L_2(\Gamma)} \langle v, \1 \rangle_{L_2(\Gamma)}$ for some $\alpha > 0$.
By comparing different values numerically, we find $\alpha = 0.05$ to give good results in our examples.

\new{With $m:={2+\ell\choose \ell}-1$, as in \S\ref{discontinuoussubspacecorrection} we equip $\V_\tria$ with a usual $L_2(\Gamma)$-orthogonal basis $\{\1|_{T}\colon T \in \tria\} \cup \{\xi_{T,i}\colon T \in \tria,\,1 \leq i \leq m\}$ where $\supp \xi_{T,i} \subset \overline{T}$, $\|\xi_{T,i}\|_{L_2(\Omega)}=|T|^{\frac12}$, and denote the resulting stiffness matrix by ${\bf A}_\tria$. The lowest order case $\ell=0$ corresponds to $m=0$.

Equipping $\Ss^{0,1}_{\tria,0}$ with the nodal basis $\Phi_\tria$ defined in \eqref{20}-\eqref{21}, for $\ell=0$ the matrix representation of the preconditioner reads as
$$
\bm{G}_\tria=\bm{D}_\tria^{-1}\big(\bm{p}_\tria^\top \bm{B}_\tria^{\Ss} \bm{p}_\tria+\beta \bm{q}_\tria^\top \bm{D}^{1/2}_\tria \bm{q}_\tria\big)\bm{D}_\tria^{-1},
$$
with $\bm{D}_\tria=\diag\{|T|\colon T \in \tria\}$ and uniformly sparse $\bm{p}_\tria$ and $\bm{q}_\tria$
as given in Sect.~\ref{impl}.

Denoting above $\bm{G}_\tria$ by $\bm{G}^0_\tria$, by applying for $\ell>0$ the subspace correction method from \S\ref{discontinuoussubspacecorrection}, the matrix representation of the resulting uniform preconditioner is given by
$$
\bm{G}_\tria=\left[\begin{array}{@{}cc@{}} \bm{G}^0_\tria & 0 \\ 0 & \beta \diag\{|T|^{-3/2} \identity_{m \times m}\colon T \in \tria\} \end{array}
\right]
$$}

The (full) matrix representations of the discretized singular integral operators $\bm{A}_\tria$ and $\bm{B}_\tria^{\Ss}$ are calculated using the BETL2 software package \cite{138.29} (alternatively, one may apply low rank approximations in a hierarchical format). Condition numbers are determined using Lanczos iteration with respect to $\lnrm\cdot\rnrm:=\|\bm{A}_\tria^{\frac{1}{2}}\cdot\|$.
\new{The constant $\beta$ is approximately optimized by comparing different choices numerically.}

We will compare our preconditioner to the diagonal preconditioner $\diag(\bm{A}_\tria)^{-1}$, and in the piecewise constant case, also to the related preconditioner
$\hat{\bm{G}}_\tria$ from \cite{138.275}, where $\hat{\bm{G}}_{\tria}=\hat{\bm{D}}_\tria^{-1} \bm{E}_\tria^\top \bm{B}_{\hat \tria}^{\Ss} \bm{E}_\tria \hat{\bm{D}}_\tria^{-\top}$ is defined as follows.
With $\hat{\tria}$ being the barycentric refinement of $\tria$, a collection $\hat{\Psi}_\tria \subset \Ss_{\hat \tria,0}^{0,1}$ is constructed in \cite{35.8655}
such that the Fortin projector $\hat{P}_\tria$ with $\ran \hat{P}_\tria=\hat{\W}_\tria:=\Span \hat{\Psi}_\tria$ and $\ran (\identity-\hat{P}_\tria)=\V_\tria^{\perp_{L_2(\Gamma)}}$ exists, and, under an additional sufficiently mildly-grading condition on the partition, has a uniformly bounded norm $\|\hat{P}_\tria\|_{\cL(\W,\W)}$ (cf.~Theorem~\ref{theorem1});
 $\hat{\bm{D}}_\tria:=\langle \Xi_\tria,\hat{\Psi}_\tria\rangle_{L_2(\Gamma)}$; $\bm{E}_\tria$ is the representation of the embedding $\hat{\W}_\tria \hookrightarrow \Ss_{\hat \tria,0}^{0,1}$ equipped with $\hat{\Psi}_\tria$ and the nodal basis of $\Ss_{\hat \tria,0}^{0,1}$, respectively; and 
$B_{\hat \tria}^{\Ss}  \in \cLis(\Ss^{0,1}_{\hat \tria},(\Ss^{0,1}_{\hat \tria})')$ is an opposite order operator that we take as 
$(B_{\hat \tria}^{\Ss} u)(v):=(Bu)(v)$, with $B$ the adapted hypersingular operator.

Compared to our $\bm{G}_\tria=\bm{G}_\tria^0$, the preconditioner $\hat{\bm{G}}_{\tria}$ has the disadvantages that, besides the aforementioned mildly grading condition, 
the matrix $\hat{\bm{D}}_{\tria}$, although uniformly sparse, is not diagonal, so that the (sufficiently accurate) application of its inverse cannot be performed in linear complexity; furthermore that it requires evaluating the adapted hypersingular operator on the larger space $\Ss_{\hat \tria,0}^{0,1} \supset \Ss_{\tria,0}^{0,1}$ ($\#\hat\tria=6\#\tria$); and finally that the non-standard barycentric refinement $\hat\tria$ has to be generated.

\subsection{Uniform refinements}
Consider a conforming triangulation $\tria_1$ of $\Gamma$ consisting of $2$ triangles per side, so $12$ triangles in total.
We let $\bbT$ be the sequence $\{\tria_k\}_{k \geq 1}$ of uniform red-refinements, where $\tria_k \succ \tria_{k-1}$ is found by subdividing each triangle from $\tria_{k-1}$ into $4$ congruent subtriangles.

For $\V_\tria=\Ss^{-1,\ell}_\tria$, Tables~\ref{tbl:unif1} and \ref{tbl:unif2} show the condition numbers of the preconditioned system for $\ell=0$ and $\ell=2$, respectively.
Aside from being uniformly bounded, the condition numbers of our preconditioner $\bm{G}_{\tria}$ are of modest size. 
In the constant case, $\ell = 0$, Table~\ref{tbl:unif1} reveals
that the preconditioner $\hat{\bm{G}}_{\tria}$ from \cite{35.8655,138.275} gives better condition numbers. As described above, this quantitative gain comes at a price. In the result of $\dim \Ss_\tria^{-1,0} = 3072$, using full matrices for the discretized
adapted hypersingular operator, we found a setup and application time of $1816$s and $0.0971$s for $\hat{\bm{G}}_{\tria}$, compared to $385$s and $0.00284$s for $\bm{G}_\tria$. These differences
are due to numerical inversion of $\hat{\bm{D}}_{\tria}$ by LU factorization with partial pivoting, and the enlargement $\Ss_{\hat \tria,0}^{0,1} \supset \Ss_{\tria,0}^{0,1}$, also causing our test machine to go out of memory in calculating $\hat{\bm{G}}_{\tria}$ for the last refinement. 
Although we expect them to be in any case significant, these differences can be made smaller when the exact inversion of $\hat{\bm{D}}_{\tria}$ is avoided, and $\bm{B}^{\Ss}_{\hat \tria}$ and $\bm{B}^{\Ss}_\tria$
are replaced by suitable low rank approximations.

 \begin{table}[]
\centering
\caption{Spectral condition numbers of the preconditioned single layer system, using uniform refinements, discretized by piecewise constants $\Ss_\tria^{-1,0}$. Both matrices $\bm{G}_\tria$ and $\hat{\bm{G}}_{\tria}$ are constructed using the adapted hypersingular operator with $\alpha=0.05$; and $\new{\beta} = 1.25$ in $\bm{G}_\tria$.}
\begin{tabular}{@{}rcccc@{}}
\toprule
dofs    & $\kappa_S(\diag(\bm{A}_\tria)^{-1}\bm{A}_\tria)$ & $\kappa_S(\bm{G}_\tria \bm{A}_\tria)$ & $\kappa_S(\hat{\bm{G}}_{\tria} \bm{A}_\tria)$\\ \midrule
$12   $ &  $14.56  $                        & $2.51$ & $1.29$  \\
$48   $ &  $29.30  $                        & $2.52$ & $1.58$  \\
$192  $ &  $58.25  $                        & $2.66$ & $1.77$  \\
$768  $ &  $116.3 $                        & $2.71$ & $1.89$  \\
$3072 $ &  $230.0 $                        & $2.74$ & $1.94$  \\
$12288$ &  $444.8 $                        & $2.79$ &   \\ \bottomrule 
\end{tabular}
\label{tbl:unif1}
\end{table}

\begin{table}[]
\centering
\caption{Spectral condition numbers of the preconditioned single layer system, using uniform refinements, discretized by discontinuous piecewise quadratics $\Ss_\tria^{-1,2}$. The matrix $\bm{G}_\tria$ is constructed using the adapted hypersingular operator, with  $\alpha=0.05$, and $\new{\beta}= 1.25$.}
\begin{tabular}{@{}rcc@{}}
\toprule
dofs    & $\kappa_S(\diag(\bm{A}_\tria)^{-1}\bm{A}_\tria)$&$\kappa_S(\bm{G}_\tria\bm{A}_\tria)$\\ \midrule
$72   $ & $167.16$                             & $9.58 $                        \\
$288  $ & $309.12$                             & $10.4$                        \\
$1152 $ & $616.03$                             & $11.1$                        \\
$4608 $ & $1211.3$                             & $11.3$                        \\
$18432$ & $2337.2$                             & $11.4$                        \\ \bottomrule
\end{tabular}
\label{tbl:unif2}
\end{table}

\subsection{Local refinements}
Here we take $\bbT$ to be the sequence $\{\tria_k\}_{k \geq 1}$ of locally refined triangulations, where $\tria_k \succ \tria_{k-1}$ is constructed using conforming newest vertex bisection to refine all triangles in $\tria_{k-1}$ that touch a corner of the cube.

As noted before, the preconditioner $\hat{\bm{G}}_{\tria}$ provides uniformly bounded condition numbers if the family $\bbT$ satisfies 
some sufficiently mildly-grading condition on the partition \cite{249.18, 138.275}. 
It is not directly clear whether $\bbT$ satisfies this condition, but we included the results nonetheless.

Table~\ref{tbl:local} gives the results for the preconditioned single layer operator discretized by piecewise constants $\Ss_\tria^{-1,0}$. 
The condition numbers $\kappa_S(\bm{G}_\tria \bm{A}_\tria)$ are nicely bounded under local refinements. In this case our preconditioner gives condition numbers slightly smaller
than the ones found with $\hat{\bm{G}}_{\tria}$. The calculation of the LU decomposition with partial pivoting of $\hat{\bm{D}}_{\tria}$ turns out to break down in the last result ($\dim \Ss_\tria^{-1,0} = 2976$).

\begin{table}[]
\centering
\caption{Spectral condition numbers of the preconditioned single layer system discretized by piecewise constants $\Ss_\tria^{-1,0}$ using local refinements at each of the eight cube corners. Both matrices $\bm{G}_\tria$ and $\hat{\bm{G}}_{\tria}$ are constructed
using the adapted hypersingular operator with $\alpha=0.05$; and $\new{\beta} = 1.2$ in $\bm{G}_\tria$. The second column is defined by $h_{\tria, min} := \min_{T \in \tria} h_T$.}
\begin{tabular}{@{}rcccc@{}}
\toprule
dofs     &  $h_{\tria,min}$   &$\kappa_S(\diag(\bm{A}_\tria)^{-1}\bm{A}_\tria) $&$\kappa_S(\bm{G}_\tria\bm{A}_\tria)$ & $\kappa_S(\hat{\bm{G}}_{\tria} \bm{A}_\tria)$\\ \midrule
$12  $   &  $7.0\cdot 10^{-1}$ & $14.56$ & $2.61$ & $1.29$  \\
$432 $   &  $2.2\cdot 10^{-2}$ & $68.66$ & $2.64$ & $2.91$  \\
$912 $   &  $6.9\cdot 10^{-4}$ & $73.15$ & $2.64$ & $3.14$  \\
$1872$   &  $6.7\cdot 10^{-7}$ & $73.70$ & $2.64$ & $3.25$  \\
$2352$   &  $2.1\cdot 10^{-8}$ & $73.80$ & $2.64$ & $3.26$  \\
$2976$   &  $2.3\cdot 10^{-10}$ & $73.66$ & $2.64$   \\ \bottomrule 
\end{tabular}
\label{tbl:local}
\end{table}

\section{Conclusion}
In this paper, we have seen how a \new{uniformly} boundedly invertible operator $B_\tria^{\Ss}$ from the space of continuous piecewise linears w.r.t.~any conforming shape regular partition $\tria$, equipped with the norm of $H^s(\Omega)$ (or $H^s(\Gamma)$) for some $s \in [0,1]$, to its dual can be used to \new{uniformly} precondition a boundedly invertible operator of opposite order discretized by discontinuous or continuous polynomials of any fixed degree w.r.t.~$\tria$. 
The cost of the resulting preconditioner is the sum of
a cost that scales linearly in $\#\tria$ and the cost 
of the application of $B_\tria^{\Ss}$.
For $\tria$ being member of a nested sequence of quasi-uniform partitions, $B_\tria^{\Ss}$ has been constructed so that it requires linear cost. \new{In a forthcoming work, we will realize this also for locally refined partitions.}
\newcommand{\etalchar}[1]{$^{#1}$}

\end{document}

%% file: weerpsi.pdf_t
\begin{picture}(0,0)%
\includegraphics{weerpsi.pdf}%
\end{picture}%
\setlength{\unitlength}{1973sp}%
\begingroup\makeatletter\ifx\SetFigFont\undefined%
\gdef\SetFigFont#1#2#3#4#5{%
  \reset@font\fontsize{#1}{#2pt}%
  \fontfamily{#3}\fontseries{#4}\fontshape{#5}%
  \selectfont}%
\fi\endgroup%
\begin{picture}(11355,3003)(3286,-3232)
\put(5476,-2011){\makebox(0,0)[lb]{\smash{{\SetFigFont{8}{9.6}{\rmdefault}{\mddefault}{\updefault}{\color[rgb]{0,0,0}$\sigma_{\tria,T}$}%
}}}}
\put(14626,-1036){\makebox(0,0)[lb]{\smash{{\SetFigFont{8}{9.6}{\rmdefault}{\mddefault}{\updefault}{\color[rgb]{0,0,0}$1\frac{1}{2}$}%
}}}}
\put(14626,-2236){\makebox(0,0)[lb]{\smash{{\SetFigFont{8}{9.6}{\rmdefault}{\mddefault}{\updefault}{\color[rgb]{0,0,0}$\frac{1}{2}$}%
}}}}
\put(14626,-3136){\makebox(0,0)[lb]{\smash{{\SetFigFont{8}{9.6}{\rmdefault}{\mddefault}{\updefault}{\color[rgb]{0,0,0}$-\frac{1}{4}$}%
}}}}
\put(3301,-436){\makebox(0,0)[lb]{\smash{{\SetFigFont{8}{9.6}{\rmdefault}{\mddefault}{\updefault}{\color[rgb]{0,0,0}$2$}%
}}}}
\put(3301,-2236){\makebox(0,0)[lb]{\smash{{\SetFigFont{8}{9.6}{\rmdefault}{\mddefault}{\updefault}{\color[rgb]{0,0,0}$\frac{1}{2}$}%
}}}}
\put(5851,-886){\makebox(0,0)[lb]{\smash{{\SetFigFont{8}{9.6}{\rmdefault}{\mddefault}{\updefault}{\color[rgb]{0,0,0}$\theta_T$}%
}}}}
\put(7951,-886){\makebox(0,0)[lb]{\smash{{\SetFigFont{8}{9.6}{\rmdefault}{\mddefault}{\updefault}{\color[rgb]{0,0,0}$\theta_{T'}$}%
}}}}
\put(11851,-1561){\makebox(0,0)[lb]{\smash{{\SetFigFont{8}{9.6}{\rmdefault}{\mddefault}{\updefault}{\color[rgb]{0,0,0}$\psi_{\tria,T}$}%
}}}}
\put(4426,-886){\makebox(0,0)[lb]{\smash{{\SetFigFont{8}{9.6}{\rmdefault}{\mddefault}{\updefault}{\color[rgb]{0,0,0}$\theta_{T''}$}%
}}}}
\end{picture}%